\documentclass[10pt, a4paper]{article} %arXiv
\usepackage{amsthm}
\usepackage[pdftex]{graphicx} 
\usepackage[pdftex]{hyperref} 
\usepackage{amsmath}
\usepackage{amssymb}
\usepackage{mathtools}
\usepackage{ascmac}
\usepackage{mathrsfs}
\usepackage{setspace}
\usepackage{nccmath}
\usepackage{color}
\usepackage{enumerate}
\usepackage[rm]{titlesec}
\usepackage{aliascnt} 
\usepackage{indentfirst}
\usepackage{tikz}
\usetikzlibrary{positioning}
\usetikzlibrary{topaths,calc}
\usepackage{here}
\usetikzlibrary{cd}
\usepackage{tikz-cd}
\usepackage{comment}

\usepackage[english]{babel}  
\addto\extrasenglish{}
\addto\extrasenglish{}
\addto\extrasenglish{}

\usepackage[margin=15mm]{geometry} 
\titleformat*{\section}{\center\large\bfseries}
\titleformat*{\subsection}{\center\bfseries}
\titleformat*{\subsubsection}{\bfseries}

\hypersetup{colorlinks=true,
citecolor=hanpurple,
linkcolor=officegreen,
urlcolor=orange}

\numberwithin{equation}{section}

\theoremstyle{plain} 
\newtheorem{thm}{Theorem}[section] 
 
\newaliascnt{lem}{thm} 
\newtheorem{lem}[lem]{Lemma}
\aliascntresetthe{lem}

\newaliascnt{prop}{thm} 
\newtheorem{prop}[prop]{Proposition}
\aliascntresetthe{prop}

\newaliascnt{cor}{thm} 
\newtheorem{cor}[cor]{Corollary}
\aliascntresetthe{cor}

\newaliascnt{conj}{thm} 
\newtheorem{conj}[conj]{Conjecture}
\aliascntresetthe{conj}

\theoremstyle{definition} 
\newaliascnt{defi}{thm} 
\newtheorem{defi}[defi]{Definition}
\aliascntresetthe{defi}

\newaliascnt{exa}{thm} 
\newtheorem{exa}[exa]{Example}
\aliascntresetthe{exa}

\newaliascnt{rem}{thm} 
\newtheorem{rem}[rem]{Remark}
\aliascntresetthe{rem}

\newaliascnt{nota}{thm} 
\newtheorem{nota}[nota]{Notation}
\aliascntresetthe{nota}

\newaliascnt{obs}{thm} 

\aliascntresetthe{obs}

\newaliascnt{property}{thm} 

\aliascntresetthe{property}

\newaliascnt{ass}{thm} 

\aliascntresetthe{ass}

\newaliascnt{setting}{thm} 

\aliascntresetthe{setting}

\newaliascnt{fact}{thm} 
\newtheorem{fact}[fact]{Fact}
\aliascntresetthe{fact}

\newcommand{\dis}{\displaystyle} 
\newcommand{\A}{\alpha}
\newcommand{\B}{\beta}

\newcommand{\CC}{\mathcal{C}}
\newcommand{\Di}{D^{-1}}
\newcommand{\dd}{\mathrm{d}}

\newcommand{\dev}{\delta_v}
\newcommand{\dex}{\delta_x}
\newcommand{\dey}{\delta_y}
\newcommand{\dez}{\delta_z}

\newcommand{\Jf}{J_\la f}
\newcommand{\K}{\kappa}
\newcommand{\KD}{\mathsf{KD}_\lambda}
\newcommand{\wKD}{\mathsf{wKD}_\lambda}

\newcommand{\la}{\lambda}
\newcommand{\LL}{\mathcal{L}}
\newcommand{\N}{\mathbb{N}}
\newcommand{\PP}{\mathscr{P}}
\newcommand{\psf}{\psi_\la f}
\newcommand{\R}{\mathbb{R}}

\newcommand{\ttilde}{\widetilde}
\newcommand{\vol}{\mathsf{vol}}
\newcommand{\Z}{\mathbb{Z}}
\newcommand{\maru}[1]{\raise0.2ex\hbox{\textcircled{\scriptsize{#1}}}} 
\newcommand{\eps}{\varepsilon} 

\newcommand{\dto}{\downarrow}
\newcommand{\rv}{\mathbb{R}^V}

\newcommand{\argmax}{\mathop{\rm arg\,max}\limits} 
\newcommand{\argmin}{\mathop{\rm arg\,min}\limits} 
\newcommand{\rwx}{m_x^\la} 
\newcommand{\rwy}{m_y^\la}

\newcommand{\kxy}{\kappa(x,y)}

\newcommand{\kLLY}{\kappa_{\mathrm{LLY}}}
\newcommand{\kIKTU}{\kappa_{\mathrm{IKTU}}}
\newcommand{\kLLYxy}{\kappa_{\mathrm{LLY}}(x,y)}
\newcommand{\kIKTUxy}{\kappa_{\mathrm{IKTU}}(x,y)}
\newcommand{\KIKTUxy}{\mathcal{K}_{\mathrm{IKTU}}(x,y)}
\newcommand{\Kxy}{\mathcal{K}(x,y)}
\newcommand{\wIKTU}{\mathsf{wIKTU}}
\newcommand{\Px}{\mathsf{P}_x}
\newcommand{\QQ}{\mathsf{Q}}
\newcommand{\Qy}{\mathsf{Q}_y}
\newcommand{\Rx}{\mathsf{R}_x}
\newcommand{\Rz}{\mathsf{R}_z}
\renewcommand{\SS}{\mathsf{S}}
\newcommand{\Sy}{\mathsf{S}_y}
\newcommand{\VV}{\mathsf{V}_0}
\newcommand{\LLP}{\LL_\mathsf{P}^0g}
\newcommand{\LLR}{\LL_\mathsf{R}^0g}
\def\:={\coloneqq} 
\def\kakko(#1){\left\langle #1 \right\rangle_{D^{-1}}}
\def\RD(#1){\mathcal{R}_{D^{-1}}(#1)}
\def\fDi(#1){{#1}_{D^{-1}}^{-1}} 
\def\flow(#1){\mathsf{b}_{#1}(f)}
\def\posiflow(#1){\mathsf{b}_{#1}^+(f)}
\def\negaflow(#1){\mathsf{b}_{#1}^-(f)}
\def\ungauss(#1){\left\lfloor #1 \right\rfloor} 
\def\upgauss(#1){\left\lceil #1 \right\rceil} 
\def\diam(#1){\mathsf{diam}(#1)}
\def\W(#1){W_1\big(#1\big)}
\def\01{\{0,1\}^{\rv}}
\def\conv(#1){\mathsf{Conv}\big( #1 \big)}
\def\Lip{\mathsf{Lip}_w^1(V)}
\def\Lipxy{\mathsf{Lip}_w^1(V;x,y)}
\def\tLip{\ttilde{\;\mathsf{Lip}_w^1}(V)\;}
\def\Lipxy{\mathsf{Lip}_w^1(V;x,y)}

\def\LIPxy{\mathsf{LIP}_w^1(V;x,y)}
\def\tLipxy{\ttilde{\;\mathsf{Lip}_w^1\;}(V;x,y)}
\def\3|(#1){|\hspace{-0.4mm}|\hspace{-0.4mm}|#1|\hspace{-0.4mm}|\hspace{-0.4mm}|_{\Di}}
\definecolor{hanpurple}{rgb}{0.32, 0.09, 0.98}
\definecolor{officegreen}{rgb}{0.0, 0.5, 0.0}
\def\red(#1){\textcolor{red}{#1}}
\def\blue(#1){\textcolor{hanpurple}{#1}}
\def\green(#1){\textcolor{officegreen}{#1}}
\def\cyan(#1){\textcolor{cyan}{#1}}
\def\orange(#1){\textcolor{orange}{#1}}

\begin{document}

\title{Weak Kantorovich difference and associated Ricci curvature \\ of hypergraphs}
\author{Tomoya Akamatsu\thanks{Department of Mathematics, Osaka University, Osaka 560-0043, Japan (\texttt{u149852g@ecs.osaka-u.ac.jp})
}}
\date{\empty}
\maketitle

\begin{abstract}
Ollivier and Lin--Lu--Yau established the theory of graph Ricci curvature (LLY curvature) via optimal transport on graphs.
Ikeda--Kitabeppu--Takai--Uehara introduced a new distance called the Kantorovich difference on hypergraphs and generalized the LLY curvature to hypergraphs (IKTU curvature).
As the LLY curvature can be represented by the graph Laplacian by M\"unch--Wojciechowski, Ikeda--Kitabeppu--Takai--Uehara conjectured that the IKTU curvature has a similar expression in terms of the hypergraph Laplacian.
In this paper, we introduce a variant of the Kantorovich difference inspired by the above conjecture and study the Ricci curvature associated with this distance ($\wIKTU$ curvature).
Moreover, for hypergraphs with a specific structure, we analyze a quantity $\CC(x,y)$ at two distinct vertices $x,y$ defined by using the hypergraph Laplacian.
If the resolvent operator converges uniformly to the identity, then $\CC(x,y)$ coincides with the $\wIKTU$ curvature along $x,y$.
\end{abstract}

%\tableofcontents

%%%%%%%%%%%%%%%%%%%%%%%%%%%%%%%%%%%%%%%%%%%%%%%%%%%%%%%%%%%%%%%%%%%%%%%%%%%
%%%%%%%%%%%%%%%%%%%%%%%%%%%%%%%%%%%%%%%%%%%%%%%%%%%%%%%%%%%%%%%%%%%%%%%%%%% 
\section{Introduction} \label{Intro}
%%%%%%%%%%%%%%%%%%%%%%%%%%%%%%%%%%%%%%%%%%%%%%%%%%%%%%%%%%%%%%%%%%%%%%%%%%%
%%%%%%%%%%%%%%%%%%%%%%%%%%%%%%%%%%%%%%%%%%%%%%%%%%%%%%%%%%%%%%%%%%%%%%%%%%%

A hypergraph is a natural generalization of a graph.
A graph describes binary relations by connecting two vertices with an edge, while in a hypergraph, one can connect three or more vertices with a hyperedge.
It is meaningful to study analysis on hypergraphs because hypergraphs can model higher-dimensional relationships, such as co-author networks.
However, many methods used for graph analysis are based on the fact that an edge is composed of two vertices and often do not work for hypergraphs.
Even if one could generalize concepts defined for graphs to general hypergraphs, it is in many cases more difficult to analyze them on hypergraphs than the case of graphs.
In this paper, we study a generalization of Ricci curvature to hypergraphs, which has been actively studied in recent years as a tool to develop analysis and geometry on graphs.

Ricci curvature is one of the most important quantities in Riemannian geometry, and its generalization to discrete spaces has been studied widely in recent years.
Although several types of discrete Ricci curvature were studied, including Forman-type and Bakry--\'{E}mery-type Ricci curvatures, in this paper, we study the \emph{LLY curvature} defined between two vertices of a graph introduced by Ollivier \cite{Ol} and Lin--Lu--Yau \cite{LLY}.
The LLY curvature of edges can be regareded as a quantity that expresses the relative ease of heat transfer in graphs, and its applications to data analysis are attracting growing interest these days (\cite{CDR, NLGGS, NLLG, SGT, SJB}, and so on).
Since hypergraphs can describe more general relationships than graphs, such quantities would also be useful for hypergraphs.
The LLY curvature $\kLLYxy$ along two vertices $x,y$ is defined by comparing the graph distance $d(x,y)$ between $x$ and $y$ with the $L^1$-Wasserstein distance $W_1\big(\rwx,\rwy\big)$ between the transition probability measures $\rwx$ and $\rwy$ at $x,y$.
We note that the key ingredients in the definition of the LLY curvature, the transition probability measures and the $L^1$-Wasserstein distance, depend on the adjacency relation of vertices.
Therefore, if we directly generalize the LLY curvature to hypergraphs, then we cannot distinguish hypergraphs from their \emph{clique expansion graphs} (see \autoref{clique expansion} for the definiton).
For this reason, generalizations of the LLY curvature to hypergraphs from different perspectives have been studied, e.g. from modified optimal transport problem on undirected or directed hypergraphs (\cite{Ak,EJ}), from the multi-marginal optimal transport problems (\cite{AGE}) and from the hypergraph Laplacian (\cite{IKTU}).
Here, the \emph{hypergraph Laplacian} $\LL$ is a multi-valued operator introduced in Yoshida \cite{Yo}.
This operator $\LL$ is derived from the heat diffusion on hypergraphs (see also \cite{HLGZ, Lo, LM}) and is applied to community detection (\cite{IMTY, TMIY}).
Ikeda--Kitabeppu--Takai--Uehara \cite{IKTU} used this hypergraph Laplacian to introduce a new distance function on the vertex set with a parameter $\la>0$, which is called the \emph{$\la$-Kantorovich difference}.
\begin{defi}[$\la$-Kantorovich difference, see \autoref{KD-def}] 
Let $\la>0$.
For two vertices $x,y$, the function $\KD(x,y)$ is defined by
\begin{equation*}
\KD(x,y)\:=\sup_{f\in\Lip}\kakko(J_\la f,\dex-\dey),
\end{equation*}
where $\Jf$ is the \emph{resolvent operator} associated with the hypergraph Laplacian, $\kakko(\cdot,\cdot)$ is the \emph{weighted inner product} and
\begin{equation*}
\Lip \:= \Big\{ f:V\to\R \;\Big|\; \kakko(f,\dex-\dey)\le d(x,y) \;\text{ for all }\;x,y\in V \Big\}.
\end{equation*}
\end{defi}
Ikeda--Kitabeppu--Takai--Uehara compared $d(x,y)$ with $\KD(x,y)$ instead of $W_1\big(\rwx,\rwy\big)$ to generalize the LLY curvature.
We call their curvature the \emph{IKTU curvature} $\kIKTUxy$.
Moreover, we can derive geometric and analytic properties of hypergraphs under suitable conditions on the IKTU curvature such as gradient estimate, Lichnerowicz-type estimate and Bonnet--Myers-type estimate (\cite[Theorems 5.2, 5.1 and 5.3]{IKTU}).
Because of its relevance with the hypergraph Laplacian, the IKTU curvature is suitable for analyzing the strength of the relationship between vertices.
For this reason, in this paper, we will focus on the IKTU curvature.

The $\la$-Kantorovich difference and the IKTU curvature are new notions and not yet well understood compared with the $L^1$-Wasserstein distance and the LLY curvature.
For example, it is difficult to calculate the IKTU curvature even for simple concrete hypergraphs.
The following conjecture by Ikeda--Kitabeppu--Takai--Uehara is one of the main motivations of our study (\cite[Remark 6.3]{IKTU}).
\begin{conj}[see \autoref{IKTU conj}]
The IKTU curvature $\kIKTUxy$ along $x,y\in V$ coincides with 
\begin{equation*}
\mathcal{C}(x,y) 
\:=
\frac{1}{d(x,y)}\inf_{f\in\Lipxy} \kakko(\LL^0f,\dex-\dey),
\end{equation*}
where $\LL^0$ is the \emph{canonical restriction} of the hypergraph Laplacian (see \autoref{canonical}), which is a single-valued operator, and
\begin{equation*}
\Lipxy \:= \Big\{ f\in\Lip \;\Big|\; \kakko(f,\dex-\dey)=d(x,y) \Big\}.
\end{equation*}
\end{conj}
A similar formula for graphs, i.e., an alternative expression of the LLY curvature via the graph Laplacian, was proved by M\"unch--Wojciechowski (\cite[Theorem 2.1]{MW}).
One can derive the inequality $\kIKTUxy\le\mathcal{C}(x,y)$ from a direct calculation (see \cite[Lemma 2.11]{KM} and \autoref{K < C}).
On the other hand, the converse inequality $\kIKTUxy\ge\mathcal{C}(x,y)$ is still open.
We divide one of the sufficient conditions for affirmatively resolving \autoref{IKTU conj} into the following two ingredients:
\begin{itemize}
\item Uniform convergence of a kind of difference function $\psf$ (see \autoref{liminf conj}),
\item Restriction of the range of the supremum in the $\la$-Kantorovich difference (see \eqref{KD restriction}).
\end{itemize}
In this paper, inspired by the second condition, we introduce the \emph{$\la$-weak Kantorovich difference} by restricting the supremum in the definition of the $\la$-Kantorovich difference as follows.
\begin{defi}[$\la$-weak Kantorovich difference, see \autoref{wKD-def}] 
Let $\la>0$.
For two vertices $x,y$, we define the function $\wKD(x,y)$ by
\begin{equation*}
\wKD(x,y)\:=\sup_{f\in\Lipxy}\kakko(\Jf,\dex-\dey). 
\end{equation*}
\end{defi}
Although it is unclear whether the $\la$-weak Kantorovich difference is always a distance function on the vertex set, from calculations for concrete hypergraphs and a comparison with the case of graphs, we expect that $\wKD(x,y)$ coincides with $\KD(x,y)$ for sufficiently small $\la>0$.
In addition, our modification seems natural also from the view of the complementary slackness (\autoref{Complementary slackness}).
As with the IKTU curvature, we define the curvature along $x,y$ by comparing $d(x,y)$ and $\wKD(x,y)$ (\autoref{IKTU-type curv}).
We call it the \emph{$\wIKTU$ curvature} $\kxy$, and then we have $\kIKTUxy\le\kxy\le\CC(x,y)$.
Moreover, the following relationship is obtained.
\begin{thm}[see \autoref{main theorem}]
If \autoref{liminf conj} is true, then the $\wIKTU$ curvature concides with $\mathcal{C}(x,y)$.
\end{thm}

By the proof of \autoref{main theorem}, we know that if \autoref{liminf conj} is true, then $\la$-weak Kantorovich potentials of $\wKD(x,y)$ are minimizers of $\mathcal{C}(x,y)$ (\autoref{main corollary}).

It is difficult to calculate the IKTU curvature, and there are only a few concrete examples where it has been calculated (\cite[Examples 6.1 and 6.4]{IKTU}).
On the other hand, the calculation of $\CC(x,y)$ can be somewhat simpler than that of the IKTU curvature thanks to the following property.
\begin{thm}[see \autoref{main theorem 2} and \autoref{flow of C}] \label{intro thm}
Let $E=\{e_V,e\}$, where $e_V$ is a hyperedge including all vertices.
In addition, let $x,y\in V$ and $f\in\Lipxy$.
If $f$ is a minimizer of $\CC(x,y)$, then $\kakko(f,\dev)=\kakko(f,\dex)$ or $\kakko(f,\dev)=\kakko(f,\dey)$ holds for every $v\in V$.
In particular, we have
\begin{equation*}
\CC(x,y)
=
\min_{f\in\LIPxy}\kakko(\LL^0f,\dex-\dey), 
\end{equation*}
where
\begin{equation*}
\LIPxy
\:=
\Big\{f\in\Lipxy\;\Big|\;\kakko(f,\delta_v)=\kakko(f,\dex)\mathrm{\;or\;}\kakko(f,\delta_v)=\kakko(f,\dey)\mathrm{\;holds\;for\;all\;}v\in V\Big\}.
\end{equation*}
\end{thm}

By \autoref{intro thm}, we can restrict the values of $f\in\Lipxy$ at vertices other than $x,y$ to two types in hypergraphs with this structure.
Hence, \autoref{intro thm} reduces the computation of $\CC(x,y)$ to a simple minimization problem.
Although we restrict ourselves to specific hypergraphs as in \autoref{intro thm}, the properties we establish in \autoref{calculation section 4} would have some generalizations to general hypergraphs. 
We believe that our arguments will be helpful for the further development of the theory of hypergraph Laplacian as well as IKTU and $\wIKTU$ curvatures.

This paper is organized as follows. 
In \autoref{Preliminaries}, we review  two types of discrete Ricci curvature, the LLY curvature of graphs and the IKTU curvature of hypergraphs.
In \autoref{IKTU conj subsection}, we discuss \autoref{IKTU conj} and the associated analysis.
In \autoref{wKD subsection} and \autoref{wIKTU subsection}, we introduce the weak Kantorovich difference and the $\wIKTU$ curvature, respectively.
In \autoref{tejun}, we discuss several useful properties for calculating $\CC(x,y)$.
In \autoref{1-regular subsection} and \autoref{|E|=2 subsection}, we describe concrete values of $\CC(x,y)$ of $1$-regular hypergraphs and hypergraphs consisting of 2 hyperedges, respectively.
In \autoref{calculation section}, we compute $\CC(x,y)$ for hypergraphs in \autoref{|E|=2 subsection}.

\bigskip
\textbf{Acknowledgements.} 
I would like to thank my supervisor Shin-ichi Ohta for his fruitful advice and discussions. 
I also would like to thank Yu Kitabeppu  for his valuable comments.
The author is supported by JST SPRING, Grant Number JPMJSP2138.

\bigskip
\textbf{Convention.} 
We denote the set of all positive and non-negative integers by $\N$ and $\N_0$, i.e., $\N\:=\Z_{\geq1}$ and $\N_{0}\:=\Z_{\geq0}$, respectively.
For $a\in\R$, denote by $\lfloor a\rfloor$ the largest integer less than or equal to $a$, and by $\lceil a\rceil$ the smallest integer greater than or equal to $a$.

%%%%%%%%%%%%%%%%%%%%%%%%%%%%%%%%%%%%%%%%%%%%%%%%%%%%%%%%%%%%%%%%%%%%%%%%%%%%
%%%%%%%%%%%%%%%%%%%%%%%%%%%%%%%%%%%%%%%%%%%%%%%%%%%%%%%%%%%%%%%%%%%%%%%%%%%%
\section{Preliminaries} \label{Preliminaries}
%%%%%%%%%%%%%%%%%%%%%%%%%%%%%%%%%%%%%%%%%%%%%%%%%%%%%%%%%%%%%%%%%%%%%%%%%%%%
%%%%%%%%%%%%%%%%%%%%%%%%%%%%%%%%%%%%%%%%%%%%%%%%%%%%%%%%%%%%%%%%%%%%%%%%%%%%

In this section, we briefly review the Ricci curvature of graphs in \cite{LLY} (\autoref{LLY subsection}) and of hypergraphs in \cite{IKTU} (\autoref{IKTU subsection}).

%%%%%%%%%%%%%%%%%%%%%%%%%%%%%%%%%%%%%%%%%%%%%%%%%%%%%%%%%%%%%%%%%%%%%%%%%%%
\subsection{Hypergraphs}
%%%%%%%%%%%%%%%%%%%%%%%%%%%%%%%%%%%%%%%%%%%%%%%%%%%%%%%%%%%%%%%%%%%%%%%%%%%

We first review hypergraphs.
A \emph{(weighted) hypergraph} $H=(V,E,w)$ is a triplet of a \emph{vertex set} $V$, a \emph{hyperedge set} $E$ and a \emph{(hyperedge) weight} $w:E\to\R_{>0}$.
A hyperedge $e\in E$ is a subset of $V$, called a \emph{(self-)loop} if $\#e=1$ and an \emph{edge} if $\#e=2$.
In particular, we call $H$ a \emph{graph} if $\#e=2$ holds for any $e\in E$.
We denote $w_e=w(e)$ for a hyperedge $e\in E$.
We say that $H$ has \emph{multi-hyperedges} if different hyperedges $e_1,e_2$ coincide as subsets of $V$.
Our hypergraphs will have no multi-hyperedges, unless otherwise noted.
In addition, our hypergraphs will be \emph{finite}, i.e. the number of vertices is finite, and let $\#V=n$.
We denote the set of all functions on the vertex set $V$ by $\rv$ and identify it with the $n$-dimensional Euclidean space.
For this reason, we often consider $f\in\rv$ as a $1\times n$ matrix.

\begin{defi} \label{clique expansion}
For a hypergraph $H=(V,E,w)$, we define the following:
\begin{itemize}
\item $x,y\in V$ are \emph{adjacent} if there exists a hyperedge $e$ such that $x,y\in e$, and the adjacency of $x,y$ is denoted by $x\sim y$.
\item 
We define the function $d:V\times V\to\R_{\ge0}$ as
\begin{align*}
d(x,y) 
\:= \min\{n\in\N_0 \;|\; x=v_0\sim\cdots\sim v_n=y \}.
\end{align*}
In this paper, we will deal only with \emph{connected} hypergraphs, i.e. $d(x,y)<\infty$ holds for all $x,y\in V$.
Then, $d$ is a distance function on the vertex set $V$, so that we can consider a hypergraph $H$ as a metric space $(V,d)$.
This distance $d$ is called the \emph{(hyper)graph distance}.
\item For $v\in V$, define $E_v\:=\{e\in E\,|\,e\ni v\}$.
\item The \emph{weighted degree} $d_v$ of $v\in V$ is defined as $d_v\:=\sum_{e\in E_v}w_e$.
\item The \emph{diameter} $\diam(H)$ of $H$ is defined as $\diam(H)\:=\max_{x,y\in V}d(x,y)$.
\item The \emph{volume} $\vol(H)$ of $H$ is defined as $\vol(H)\:=\sum_{v\in V}d_v$.
\item
The \emph{clique expansion graph} associated with $H$ is the graph $G=(V,E_G)$
\footnote{We do not define a weight of a clique expansion since its definition is not unique.}
such that $(x,y)\in E_G$ if and only if $x\neq y$ and $(x,y)\in e$ for some $e\in E$.
We remark that the clique expansion graphs of different hypergraphs can coincide.
\end{itemize}
\end{defi}

%%%%%%%%%%%%%%%%%%%%%%%%%%%%%%%%%%%%%%%%%%%%%%%%%%%%%%%%%%%%%%%%%%%%%%%%%%%
\subsection{Ricci curvature of graphs} \label{LLY subsection}
%%%%%%%%%%%%%%%%%%%%%%%%%%%%%%%%%%%%%%%%%%%%%%%%%%%%%%%%%%%%%%%%%%%%%%%%%%%

Let $G=(V,E,w)$ be a (weighted) finite graph\footnote{Since we consider only finite hypergraphs in this paper, our graphs will also be finite throughout the discussion of graphs.}.
We first define the transition probability measure at each vertex.
We denote the set of all probability measures on $V$ by $\PP(V)$.

\begin{defi} \label{random walk}
For each vertex $x\in V$ and $\la\in[0,1]$, we define a probability measure $\rwx\in\PP(V)$ as follows:
\begin{equation*}
\dis \rwx(y) 
\:=
\left\{
\begin{aligned}
&\;1-\la & (y=x) , \vspace{1mm} \\
&\dis \;\la\cdot\frac{w_{xy}}{d_x} & (y\sim x) , \vspace{1mm} \\
&\;0 & (\text{otherwise}).
\end{aligned}
\right. 
\end{equation*}
\end{defi}

The LLY curvature along vertices $x,y$ is defined by comparing the graph distance $d(x,y)$ between $x$ and $y$ and the $L^1$-Wasserstein distance between the transition probability measures $\rwx$ and $\rwy$ at them.

\begin{defi}[$L^1$-Wasserstein distance]
We define the \emph{$L^1$-Wasserstein distance} $W_1(\mu,\nu)$ between $\mu,\nu\in\PP(V)$ as 
\begin{equation}
W_1(\mu,\nu) 
\:=
\sup \left\{ \sum_{v\in V}f(v) \big( \mu(v)-\nu(v) \big) 
\;\middle|\;
f\in\mathsf{Lip}^1(V) \right\}, \label{W1-dist}
\end{equation}
where $\mathsf{Lip}^1(V)$ is the set of all $1$-Lipschitz functions on $(V,d)$.
\end{defi}

\begin{rem}
The $L^1$-Wasserstein distance $W_1$ is originally defined by the minimization of the transport cost and equality \eqref{W1-dist} holds by the Kantorovich--Rubinstein duality of $W_1$.
In this paper, we do not discuss optimal transport theory and will utilize the form of \eqref{W1-dist} later on, so we defined $W_1$ by \eqref{W1-dist}.
\end{rem}

\begin{defi}[Lin--Lu--Yau curvature of graphs {\cite{LLY}}]
The \emph{LLY curvature} $\kLLYxy$ along two vertices $x,y$ is defined as
\begin{equation*}
\kLLYxy 
\:=
\lim_{\la\dto0} \frac{1}{\la} \left(\dis 1-\frac{W_1\big(\rwx,\rwy\big)}{d(x,y)} \right).
\end{equation*}
The limit in the right-hand side exists.
\end{defi}

%%%%%%%%%%%%%%%%%%%%%%%%%%%%%%%%%%%%%%%%%%%%%%%%%%%%%%%%%%%%%%%%%%%%%%%%%%%
\subsection{Ricci curvature of hypergraphs} \label{IKTU subsection}
%%%%%%%%%%%%%%%%%%%%%%%%%%%%%%%%%%%%%%%%%%%%%%%%%%%%%%%%%%%%%%%%%%%%%%%%%%%

In this subsection, we explain how to generalize the LLY curvature of graphs to hypergraphs.
As discussed in \autoref{LLY subsection}, the definition of the LLY curvature of graphs uses $W_1\big(\rwx,\rwy\big)$, i.e. the transition probability measures and the optimal transport cost on graphs.
However, on the one hand, random walks and optimal transport problem on hypergraphs are nontrivial; 
especially it is difficult to distinguish a hypergraph with its clique expansion from these respects.
On the other hand, via the \emph{graph Laplacian}
\begin{equation*}
\Delta:\rv\to\rv;\quad
\Delta f(x)\:=\frac{1}{d_x}\sum_{y\sim x}w_{xy}\big(f(x)-f(y)\big),
\end{equation*}
we have a deformation
\begin{equation}
W_1\big(\rwx,\rwy\big) 
= 
\sup_{f\in\mathsf{Lip}^1(V)}\big\langle(I-\la\Delta)f,\dex-\dey\big\rangle
= 
\sup_{f\in\mathsf{Lip}^1(V)}\big\langle(I+\la\Delta)^{-1}f,\dex-\dey\big\rangle+o(\la), \label{wxy formula}
\end{equation}
where $I:\rv\to\rv$ is the identity operator, $\dex\in\rv$ is the characteristic function at $x\in V$ and $\langle\cdot,\cdot\rangle$ is the canonical inner product on $\rv$, i.e. 
\begin{equation*}
\dis \dex(v) \:=
\left\{
\begin{aligned}
&\;1 & (v=x) , \vspace{1mm} \\
&\;0 & (v\neq x),
\end{aligned}
\right. 
\quad
\text{and}
\quad
\langle f_1,f_2\rangle\:=\sum_{v\in V}f_1(v)f_2(v).
\end{equation*}
We focus on the main term of \eqref{wxy formula} to consider the limit $\la\dto0$ in the study of the LLY curvature $\kLLYxy$.
On a hypergraph, Ikeda--Kitabeppu--Takai--Uehara \cite{IKTU} utilized this idea to give an appropriate modification of $W_1\big(\rwx,\rwy\big)$, and the function on $V\times V$ thus obtained is a distance function on a hypergraph.
This new distance function on a hypergraph is called the \emph{Kantorovich difference}.
Then, the LLY curvature is generalized to hypergraphs by comparing the Kantorovich difference to the (hyper)graph distance.

%%%%%%%%%%%%%%%%%%%%%%%%%%%%%%%%%%%%%%%%%%%%%%%%%%%%%%%%%%%%%%%%%%%%%%%%%%%
\subsubsection{Kantorovich difference}
%%%%%%%%%%%%%%%%%%%%%%%%%%%%%%%%%%%%%%%%%%%%%%%%%%%%%%%%%%%%%%%%%%%%%%%%%%%

First, we recall the definition of the hypergraph Laplacian and its basic properties to define the Kantorovich difference.
We refer to \cite{IKTU} for further details.

\begin{defi}[Weighted inner product] \label{weighted inner product}
We define the \emph{weighted inner product} $\kakko(\cdot,\cdot):\rv\times\rv\to\R$ and the \emph{norm} $\|\cdot\|_{\Di}:\rv\to\R_{\ge0}$ on $\rv$ as 
\begin{equation*}
\kakko(f,g) 
\:=
\sum_{v\in V}\frac{f(v)g(v)}{d_v},
\quad
\|f\|_{D^{-1}}\:=\sqrt{\kakko(f,f)}.
\end{equation*}
\end{defi}

Note that $\big(\rv,\kakko(\cdot,\cdot)\big)$ is a Hilbert space.
We define the base polytope of a hyperedge for the definition of the hypergraph Laplacian.

\begin{defi}[Base polytopes of hyperedges] \label{conv}
For a hyperedge $e$, we define its \emph{base polytope} $\mathsf{B}_e\subset\rv$ as
\begin{equation*}
\mathsf{B}_e
\:=
\conv(\{\dex-\dey\;|\;x,y\in e\}),
\end{equation*}
where $\conv(A)$ is the \emph{convex hull} of $A\subset\rv$ in $\rv$.
\end{defi}

\begin{defi}[Hypergraph Laplacian] \label{LL}
We define the \emph{(normalized) hypergraph Laplacian} by the multi-valued operator
\begin{equation*}
\LL:\rv\to\01;\quad
\LL f=\LL(f)\:=\left\{ \sum_{e\in E}w_e\big\langle f,\flow(e)\big\rangle_{\Di}\, \flow(e) \;\middle|\; \flow(e)\in\argmax_{\mathsf{b}\in \mathsf{B}_e}\kakko(f,\mathsf{b})\right\}
\end{equation*}
on the Hilbert space $(\rv,\kakko(\cdot,\cdot))$.
\end{defi}

This hypergraph Laplacian is derived from the Laplacian introduced in a more general framework in \cite{Yo} (see also \cite{LM}).

\begin{rem}[Nonlinearity]
Note that, on the one hand, we do not necessarily have $\LL(f_1+f_2)=\LL(f_1)+\LL(f_2)$ for $f_1,f_2\in\rv$.
On the other hand, $\LL(cf)=c\LL(f)$ holds for any $c\in\R$ and $f\in\rv$ from the definition.
\end{rem}

It is known that $\LL$ is a \emph{maximal monotone operator} on the Hilbert space $\big(\rv,\kakko(\cdot,\cdot)\big)$, which implies the following.

\begin{prop}[{\cite[Lemma 2.15]{Mi}}] \label{closed convex}
For all $f\in\rv$, $\LL f\subset\rv$ is closed and convex.
\end{prop}

\begin{defi}[Resolvent] \label{resolvent}
For $\la>0$, we define the \emph{resolvent operator} $J_\la:\rv\to\01$ as
\begin{equation*}
\Jf=J_\la(f)\:=(I+\la\LL)^{-1}(f).
\end{equation*}
\end{defi}

\begin{prop}[{\cite[Lemma 2.1]{IKTU}}]
The operator $J_\la$ is single-valued and continuous.
Moreover, $\LL$ is the subdifferential of the convex function 
\begin{equation*}
\mathcal{E}:\rv\to\R;\quad
f\mapsto\frac{1}{2}\sum_{e\in E}w_e\max_{x,y\in e}\kakko(f,\delta_x-\delta_y)^2,
\end{equation*}
and it follows that (see also \cite[Remark 2.4]{KM}):
\begin{equation*}
J_\la f=\argmin_{g\in\rv}\left\{ \frac{\|f-g\|_{D^{-1}}^2}{2\la}+\mathcal{E}(g) \right\}. %\label{resolvent-argmin}
\end{equation*}
\end{prop}

\begin{defi}[Weighted $1$-Lipschitz function] \label{weighted 1-Lip}
We say that $f\in\rv$ is \emph{weighted $1$-Lipschitz} if it satisfies $\kakko(f,\dex-\dey)\le d(x,y)$ for all $x,y\in V$.
We denote the set of all weighted $1$-Lipschitz functions on $V$ by $\Lip$ and define its subset $\tLip$ as
\begin{equation*}
\tLip\:=\left\{ f\in\Lip \;\middle|\; \max_{v\in V}\kakko(f,\delta_v)\le\diam(H) \right\}.
\end{equation*}
\end{defi}

\begin{defi}[$\la$-Kantorovich difference] \label{KD-def}
Let $\la>0$.
We define the \emph{$\la$-Kantorovich difference} $\KD(x,y)$ between two vertices $x,y$ as
\begin{equation}
\KD(x,y)\:=\sup_{f\in\Lip}\kakko(J_\la f,\dex-\dey). \label{eq-KD-def}
\end{equation}
\end{defi}

There exists $f_\la\in\Lip$ attaining the supremum in the right-hand side for any $\la>0$ (\cite[Proposition 3.7]{IKTU}).
Then $f_\la$ is called a \emph{$\la$-Kantorovich potential} of $\KD(x,y)$.
We remark that a $\la$-Kantorovich potential is not always unique.

\begin{prop}[{\cite[Propositions 3.5 and 3.3]{IKTU}}]
For any $\la>0$, the following hold:
\begin{itemize}
\item[$(1)$] The $\la$-Kantorovich difference is a distance function on $V$.
\item[$(2)$] The supremum in the right-hand side of \eqref{eq-KD-def} is unchanged when we restrict $\Lip$ to $\tLip$.
\end{itemize}
\end{prop}

%%%%%%%%%%%%%%%%%%%%%%%%%%%%%%%%%%%%%%%%%%%%%%%%%%%%%%%%%%%%%%%%%%%%%%%%%%%
\subsubsection{Ikeda--Kitabeppu--Takai--Uehara curvature}
%%%%%%%%%%%%%%%%%%%%%%%%%%%%%%%%%%%%%%%%%%%%%%%%%%%%%%%%%%%%%%%%%%%%%%%%%%%

\begin{defi}[Ikeda--Kitabbepu--Takai--Uehara curvature of hypergraphs \cite{IKTU}]
The \emph{IKTU curvature} $\kIKTUxy$ along two vertices $x,y$ is defined as
\begin{equation*}
\kIKTUxy\:=\lim_{\la\dto0}\frac{1}{\la}\left( 1-\frac{\KD(x,y)}{d(x,y)} \right).
\end{equation*}
The limit in the right-hand side exists (\cite[Theorem A.1]{IKTU}).
\end{defi}

The LLY curvature and the IKTU curvature coincide for graphs.

\begin{thm}[{\cite[Proposition 4.1]{IKTU}}]
When $H$ is a graph, $\kLLYxy=\kIKTUxy$ holds for any vertices $x,y$.
\end{thm}

%%%%%%%%%%%%%%%%%%%%%%%%%%%%%%%%%%%%%%%%%%%%%%%%%%%%%%%%%%%%%%%%%%%%%%%%%%%
%%%%%%%%%%%%%%%%%%%%%%%%%%%%%%%%%%%%%%%%%%%%%%%%%%%%%%%%%%%%%%%%%%%%%%%%%%%
\section{Weak Kantorovich difference and associated Ricci curvature} \label{wKD}
%%%%%%%%%%%%%%%%%%%%%%%%%%%%%%%%%%%%%%%%%%%%%%%%%%%%%%%%%%%%%%%%%%%%%%%%%%%
%%%%%%%%%%%%%%%%%%%%%%%%%%%%%%%%%%%%%%%%%%%%%%%%%%%%%%%%%%%%%%%%%%%%%%%%%%%

In this section, we propose a restriction on the range of the supremum in the definition of the $\la$-Kantorovich difference.
We first introduce a conjectured alternative expression of the IKTU curvature.

%%%%%%%%%%%%%%%%%%%%%%%%%%%%%%%%%%%%%%%%%%%%%%%%%%%%%%%%%%%%%%%%%%%%%%%%%%%
\subsection{Ikeda--Kitabeppu--Takai--Uehara's conjecture} \label{IKTU conj subsection}
%%%%%%%%%%%%%%%%%%%%%%%%%%%%%%%%%%%%%%%%%%%%%%%%%%%%%%%%%%%%%%%%%%%%%%%%%%%

\begin{defi}
For $f\in\rv$, we define 
\begin{equation*}
\3|(\LL f)\:=\inf_{f^\prime\in\LL f}\|f^\prime\|_{D^{-1}}.
\end{equation*}
\end{defi}

The above infimum is attained for any $f\in\rv$ by \autoref{closed convex}, and the minimizer is unique.
This allows us to restrict $\LL$ into a single-valued operator.

\begin{defi}[Canonical restriction of $\LL$] \label{canonical}
$\LL^0:\rv\to\rv$ defined by $\LL^0(f)=\LL^0f\in\LL f$ and $\|\LL^0f\|_{D^{-1}}=\3|(\LL f)$ is called the \emph{canonical restriction} of $\LL$.
\end{defi}

\begin{fact}[{\cite[Lemma 2.22 and Theorem 3.5]{Mi}}]
$\LL^0f$ can also be represented as
\begin{equation}
\LL^0f=\lim_{\la\dto0}\frac{f-J_\la f}{\la}. \label{L0-limit}
\end{equation}
\end{fact}

\begin{defi}
For two vertices $x,y$, we define the subset $\Lipxy\subset\Lip$ as
\begin{equation*}
\Lipxy \:= \Big\{ f\in\Lip \;\Big|\; \kakko(f,\dex-\dey)=d(x,y) \Big\}.
\end{equation*}
Furthermore, we also define
\begin{equation*}
\mathcal{C}(x,y) \:= \frac{1}{d(x,y)}\inf_{f\in\Lipxy} \kakko(\LL^0f,\dex-\dey).
\end{equation*}
\end{defi}

Ikeda--Kitabeppu--Takai--Uehara conjectured that the following formula holds for the IKTU curvature.

\begin{conj}[Ikeda--Kitabeppu--Takai--Uehara's conjecture {\cite[Remark 6.3]{IKTU}}] \label{IKTU conj}
For any $x,y\in V$, it holds that
\begin{equation*}
\kIKTUxy = \mathcal{C}(x,y). 
\end{equation*}
\end{conj}

When $H$ is a graph, we know that a similar relationship holds for the LLY curvature and the graph Laplacian (\cite[Theorem 2.1]{MW}).

One can see by direct computation the following.

\begin{prop}[see also {\cite[Lemma 2.11]{KM}}] \label{K < C}
For any $x,y\in V$, we have 
\begin{equation*}
\kIKTUxy \le \mathcal{C}(x,y). 
\end{equation*}
\end{prop}

\begin{proof}
Notice that, by replacing $f$ with $f-f(y)$, we can restrict the range of the infimum of $\CC(x,y)$ to 
\begin{equation*}
\tLipxy \:= \bigg\{ f\in\tLip \;\bigg|\; \kakko(f,\dex-\dey)=d(x,y) \bigg\}.
\end{equation*}
Combining this with \eqref{L0-limit}, we have
\begin{align}
d(x,y)\CC(x,y)
&= 
\inf_{f\in\tLipxy}\kakko(\lim_{\la\dto0}\frac{f-\Jf}{\la},\dex-\dey) & \notag \\
&= 
\inf_{f\in\tLipxy}\lim_{\la\dto0}\frac{d(x,y)-\kakko(\Jf,\dex-\dey)}{\la} & \label{inf,lim} \\
&= 
\lim_{\la\dto0}\frac{d(x,y)-\kakko(J_\la g,\dex-\dey)}{\la} &\Big( g\in\tLipxy:\text{ a minimizer} \Big)\notag \\
&\ge \lim_{\la\dto0}\frac{1}{\la}\Bigg( d(x,y)-\sup_{f\in\tLipxy}\kakko(\Jf,\dex-\dey) \Bigg) & \label{lim,-sup} \\
&\ge \lim_{\la\dto0}\frac{1}{\la}\Bigg( d(x,y)-\sup_{f\in\tLip}\kakko(\Jf,\dex-\dey) \Bigg) & \notag \\
&= d(x,y)\kIKTUxy. & \notag 
\end{align}
We remark that the limit in \eqref{lim,-sup} exists (see \autoref{wIKTU subsection}).
\end{proof}

The other inequality
\begin{equation}
\kIKTUxy\ge\mathcal{C}(x,y) \label{open ineq}
\end{equation}
is open.
Now, we have a closer look into \eqref{inf,lim} and \eqref{lim,-sup}.

\begin{defi} \label{phi psi def}
For $\la>0$ and $f\in\rv$, we define $\psf\in\rv$ by $\psf(v)\:=\big(f(v)-\Jf(v)\big)-\LL^0f(v)\cdot\la$.
\end{defi}

Note that $\lim_{\la\dto0}\psf(v)/\la=0$ holds by \eqref{L0-limit}.

\begin{lem} \label{key prop}
Let $\la\in[0,1]$ and $f\in\Lip$.
For any $x,y\in V$, we have
\begin{equation*}
\KD(x,y)
= 
\sup_{f\in\tLip}\Big\{\kakko(f,\dex-\dey)-\kakko(\LL^0f,\dex-\dey)\la-\kakko(\psf,\dex-\dey)\Big\}. 
\end{equation*}
\end{lem}

\begin{proof}
The claim follows from direct computations:
\begin{equation}
\kakko(\Jf,\dex-\dey) 
= 
\kakko(f,\dex-\dey)-\big\{\kakko(\LL^0f,\dex-\dey)\la+\kakko(\psf,\dex-\dey)\big\}. \label{pf-la formula 2}
\end{equation}
\end{proof}

Using \eqref{pf-la formula 2}, we can rewrite \eqref{inf,lim} and \eqref{lim,-sup} as 
\begin{align*}
\eqref{inf,lim}
&=
\inf_{f\in\tLipxy}\lim_{\la\dto0}\frac{1}{\la}\Bigg( d(x,y)-\Big\{\kakko(f,\dex-\dey)-\kakko(\LL^0f,\dex-\dey)\la-\kakko(\psf,\dex-\dey)\Big\} \Bigg) \\
&=
\inf_{f\in\tLipxy}\lim_{\la\dto0}\Bigg\{\kakko(\LL^0f,\dex-\dey)+\frac{\dis\kakko(\psf,\dex-\dey)}{\la}\Bigg\}, \\
\eqref{lim,-sup}
&=
\lim_{\la\dto0}\frac{1}{\la}\Bigg( d(x,y)-\sup_{f\in\tLipxy}\Big\{\kakko(f,\dex-\dey)-\kakko(\LL^0f,\dex-\dey)\la-\kakko(\psf,\dex-\dey)\Big\} \Bigg) \\
&=
\lim_{\la\dto0}\inf_{f\in\tLipxy}\Bigg\{\kakko(\LL^0f,\dex-\dey)+\frac{\dis\kakko(\psf,\dex-\dey)}{\la}\Bigg\}.
\end{align*}

\begin{rem} \label{key remark}
Let $f_\la$ be a $\la$-Kantorovich potential of $\KD(x,y)$.
Then, $\kakko(\LL^0f_\la,\dex-\dey)$ is a constant for sufficiently small $\la>0$ thanks to \cite[Proposition A.4]{IKTU}.
\end{rem}

\begin{nota} \label{nota}
We denote the constant given in \autoref{key remark} by $\KIKTUxy$:
For sufficiently $\la>0$,
\begin{equation*}
\KIKTUxy\:=\kakko(\LL^0f_\la,\dex-\dey).
\end{equation*}
\end{nota}

We expect the following to hold.

\begin{conj} \label{liminf conj}
For any $x,y\in V$, it holds that
\begin{equation*}
\lim_{\la\dto0}\inf_{f\in\tLip}\frac{\kakko(\psf,\dex-\dey)}{\la}
=
\inf_{f\in\tLip}\lim_{\la\dto0}\frac{\kakko(\psf,\dex-\dey)}{\la}
\quad(=0). 
\end{equation*}
\end{conj}

If \autoref{liminf conj} is true, then $\kIKTUxy$ can be estimated from below by $\KIKTUxy$.

\begin{prop} \label{kIKTU lower bound}
If \autoref{liminf conj} is true, then for any $x,y\in V$, we have
\begin{equation*}
\kIKTUxy \ge \frac{\KIKTUxy}{d(x,y)}. 
\end{equation*}
\end{prop}

\begin{proof}
Using \eqref{pf-la formula 2}, we can estimate as follows.
\begin{align*}
&\;\quad 
d(x,y)\kIKTUxy & \\
&= 
\lim_{\la\dto0}\frac{1}{\la}\Bigg\{ d(x,y)-\bigg(\kakko(f_\la,\dex-\dey)-\kakko(\LL^0f_\la,\dex-\dey)\la-\kakko(\psf_\la,\dex-\dey)\bigg) \Bigg\} & \\
&\ge 
\limsup_{\la\dto0} \frac{1}{\la}\bigg(\kakko(\LL^0f_\la,\dex-\dey)\la+\kakko(\psf_\la,\dex-\dey)\bigg) & \\
&\ge 
\lim_{\la\dto0}\Bigg(\kakko(\LL^0f_\la,\dex-\dey)+\inf_{f\in\tLip}\frac{\dis\kakko(\psf,\dex-\dey)}{\la}\Bigg) & \\
&= 
\KIKTUxy+\lim_{\la\dto0}\inf_{f\in\tLip}\frac{\kakko(\psf,\dex-\dey)}{\la} & \\
&=
\KIKTUxy, &\Big(\text{by \autoref{liminf conj}}\Big)
\end{align*}
where $f_\la$ is a $\la$-Kantorovich potential of $\KD(x,y)$.
\end{proof}

By combining this with \autoref{K < C}, $\kIKTUxy$ is estimated from above and below as follows.

\begin{cor}
If \autoref{liminf conj} is true, then for any $x,y\in V$, we have
\begin{equation*}
\frac{\KIKTUxy}{d(x,y)} 
\le
\kIKTUxy 
\le 
\inf_{f\in\tLipxy} \frac{\kakko(\LL^0f,\dex-\dey)}{d(x,y)} 
\quad \Big(=\mathcal{C}(x,y)\Big). 
\end{equation*}
\end{cor}

Hence, if $f_\la\in\Lipxy$ for sufficiently small $\la>0$, then \autoref{IKTU conj} is true.

\begin{thm} \label{main thm}
If \autoref{liminf conj} is true and for sufficiently small $\la>0$,
\begin{equation}
\KD(x,y) = \sup_{f\in\tLipxy}\kakko(J_\la f,\dex-\dey) \label{KD restriction}
\end{equation}
holds, then \autoref{IKTU conj} is true.
\end{thm}

\begin{proof}
If \eqref{KD restriction} holds, then by the definition of $\mathcal{K}_{\mathrm{IKTU}}(x,y)$, we have
\begin{equation*}
\frac{\KIKTUxy}{d(x,y)} 
\ge
\inf_{f\in\tLipxy} \frac{\kakko(\LL^0f,\dex-\dey)}{d(x,y)}.
\end{equation*}
Thus, we conclude that \autoref{IKTU conj} is true.
\end{proof}

%%%%%%%%%%%%%%%%%%%%%%%%%%%%%%%%%%%%%%%%%%%%%%%%%%%%%%%%%%%%%%%%%%%%%%%%%%%
\subsection{Weak Kantorovich difference} \label{wKD subsection}
%%%%%%%%%%%%%%%%%%%%%%%%%%%%%%%%%%%%%%%%%%%%%%%%%%%%%%%%%%%%%%%%%%%%%%%%%%%

Now, we consider a modification of the Kantorovich difference by restricting the range of the supremum as in \eqref{KD restriction}, based on the discussion in \autoref{IKTU conj subsection}. 

\begin{defi}[$\la$-weak Kantorovich difference] \label{wKD-def}
Let $\la>0$.
We define the \emph{$\la$-weak Kantorovich difference} $\wKD(x,y)$ between two vertices $x,y$ as 
\begin{equation}
\wKD(x,y)\:=\sup_{f\in\Lipxy}\kakko(\Jf,\dex-\dey). \label{eq-wKD-def}
\end{equation}
\end{defi}

Note that $\KD(x,y)\ge\wKD(x,y)$ holds by definition.
One can restrict the range of the supremum of the right-hand side of \eqref{eq-wKD-def} to $\tLipxy$ from the same proof as that for \cite[Proposition 3.3]{IKTU}.

\begin{prop} \label{distance}
The following hold:
\begin{itemize}
\item[$(1)$] $\wKD(x,y)\ge0$ holds for any $\la>0$ and $x,y\in V$.
Moreover, for sufficiently small $\la>0$, $\wKD(x,y)=0$ holds if and only if $x=y$.
\item[$(2)$] $\wKD(x,y)=\wKD(y,x)$ holds for any $\la>0$ and $x,y\in V$.
\item[$(3)$] $\wKD(x,z)\le\wKD(x,y)+\wKD(y,z)$ holds for any $\la>0$ and $x,y,z\in V$ with $d(x,z)=d(x,y)+d(y,z)$.
\end{itemize}
\end{prop}

\begin{proof} 
\hspace{0mm}
\begin{itemize}
\item[$(1)$] 
By definition, we have $\wKD(x,y)\ge0$ and $\wKD(x,x)=0$ for any $x,y\in V$.
Moreover, if $x\neq y$, then $\wKD(x,y)>0$ holds for sufficiently small $\la>0$ since $\lim_{\la\dto0}\Jf=f$ implies $\kakko(\Jf,\dex-\dey)>0$ for $f\in\Lipxy$.
Since $V$ is finite, we can choose $\la>0$ such that $\wKD(x,y)>0$ for all $x,y\in V$ with $x\neq y$.
\item[$(2)$] 
This is trivial by the definition.
\item[$(3)$] 
For any $\eps>0$, there exists a function $f_{xz}\in\mathsf{Lip}_w^1(V;x,z)$ such that 
\begin{equation*}
\wKD(x,z)
\le
\kakko(\Jf_{xz},\dex-\delta_z)+\eps
=
\kakko(\Jf_{xz},\dex-\dey)+\kakko(\Jf_{xz},\dey-\delta_z)+\eps.
\end{equation*}
Now, since $f_{xz}\in\Lipxy\cap\mathsf{Lip}_w^1(V;y,z)$ by the assumption, we find
\begin{equation*}
\kakko(\Jf_{xz},\dex-\dey) \le \wKD(x,y),
\quad
\kakko(\Jf_{xz},\dey-\delta_z) \le \wKD(y,z).
\end{equation*}
Thus, $\wKD(x,z)\le\wKD(x,y)+\wKD(y,z)$ holds since $\eps>0$ was arbitrary.
\end{itemize}
\end{proof}

It is unclear whether the weak Kantorovich difference always satisfies the triangle inequality.

\begin{prop}
For any $\la>0$ and $x,y\in V$, there exists a function $f_\la\in\Lipxy$ which attains the supremum of the right-hand side of \eqref{eq-wKD-def}.
\end{prop}

\begin{proof}
This follows from the same proof as that for \cite[Proposition 3.7]{IKTU}.
\end{proof}

We call this $f_\la\in\Lipxy$ a \emph{$\la$-weak Kantorovich potential} of $\wKD(x,y)$.

We finally discuss the validity of the definition of the weak Kantorovich difference in respect of its relation to the $L^1$-Wasserstein distance on graphs.
Let us first recall the relation between the Kantorovich difference and the $L^1$-Wasserstein distance.

\begin{fact}[see \eqref{wxy formula}, {\cite[Equation (4.2)]{IKTU}}]
Let $G=(V,E)$ be a graph.
Then we have, as $\la\dto0$,
\begin{equation*}
W_1\big( \rwx,\rwy \big)
= \KD(x,y) + o(\la). 
\end{equation*}
\end{fact}

Furthermore, for the range of the supremum in the right-hand side in \eqref{W1-dist}, the following is known.

\begin{fact}[The complementary slackness {\cite[Lemma 3.1]{BCLMP}}] \label{Complementary slackness}
Let $G=(V,E)$ be a graph, $x,y\in V$ and $\la\in[0,1]$.
Then, any maximizer $f_\la^{W_1}\in\mathsf{Lip}^1(V)$ attaining $W_1\big(m_x^\la,m_y^\la\big)$ satisfies 
\begin{equation*}
f_\la^{W_1}(x)-f_\la^{W_1}(y)=d(x,y). 
\end{equation*}
\end{fact}

Although there is not a clear relation between the potentials $f_\la^{W_1}$ for $W_1\big(\rwx,\rwy\big)$ and $f_\la$ for $\KD(x,y)$, we expect that the $\la$-Kantorovich potential $f_\la$ satisfies $f_\la\in\Lipxy$, namely the Kantorovich difference and the weak Kantorovich difference coincide.
This observation motivated our introduction of the weak Kantorovich difference $\wKD(x,y)$.
In addition, the weak Kantorovich difference has the merit of better computability than the Kantorovich difference.

%%%%%%%%%%%%%%%%%%%%%%%%%%%%%%%%%%%%%%%%%%%%%%%%%%%%%%%%%%%%%%%%%%%%%%%%%%%
\subsection{$\wIKTU$ curvature} \label{wIKTU subsection}
%%%%%%%%%%%%%%%%%%%%%%%%%%%%%%%%%%%%%%%%%%%%%%%%%%%%%%%%%%%%%%%%%%%%%%%%%%%

We conclude this section by defining the $\wIKTU$ curvature associated with the weak Kantorovich difference.

\begin{defi}[Weak Ikeda--Kitabeppu--Takai--Uehara curvature of hypergraphs] \label{IKTU-type curv}
We define the \emph{weak Ikeda--Kitabeppu--Takai--Uehara curvature} (\emph{$\wIKTU$ curvature}) $\kxy$ along two vertices $x,y$  as
\begin{equation*}
\kxy\:=\lim_{\la\dto0}\frac{1}{\la}\left( 1-\frac{\wKD(x,y)}{d(x,y)} \right).
\end{equation*}
\end{defi}

The limit in the right-hand side exists by the same proof as that for \cite[Theorem A.1]{IKTU}.
By the same discussion as \cite[Proposition 4.1]{IKTU}, thanks to \autoref{Complementary slackness}, the LLY curvature and the $\wIKTU$ curvature coincide for graphs.
Note that from the relation between $\KD(x,y)$ and $\wKD(x,y)$, for any two vertices $x,y$, 
\begin{equation}
\kIKTUxy\le\kxy. \label{kIKTU < k}
\end{equation}
Notice that if \eqref{KD restriction} holds, then this inequality becomes equality, that is, $\kIKTUxy=\kxy$ holds.

By the discussion in \autoref{IKTU conj subsection}, we can show an analogue of \autoref{IKTU conj} provided that \autoref{liminf conj} is true.

\begin{thm} \label{main theorem}
If \autoref{liminf conj} is true, then for any $x,y\in V$, it holds that
\begin{equation*}
\kxy=\CC(x,y).
\end{equation*}
\end{thm}

\begin{proof}
We note that $\la$-weak Kantorovich potentials of $\wKD(x,y)$ have the same property as \autoref{key remark}.
Thus, for sufficiently small $\la>0$, we can denote $\kakko(\LL^0f_\la,\dex-\dey)$ by a constant $\Kxy$ that does not depend on $\la$ as with \autoref{nota}.
Then, on the one hand, in the same way as \autoref{kIKTU lower bound}, we obtain
\begin{align}
d(x,y)\kxy 
&= 
\lim_{\la\dto0}\Bigg(\kakko(\LL^0f_\la,\dex-\dey)+\frac{\dis\kakko(\psf_\la,\dex-\dey)}{\la}\Bigg) & \notag \\
&\ge 
\lim_{\la\dto0}\Bigg(\kakko(\LL^0f_\la,\dex-\dey)+\inf_{f\in\tLip}\frac{\dis\kakko(\psf,\dex-\dey)}{\la}\Bigg) & \notag \\
&= 
\Kxy+\lim_{\la\dto0}\inf_{f\in\tLip}\frac{\kakko(\psf,\dex-\dey)}{\la} & \notag \\
&=
\Kxy, &\Big(\text{by \autoref{liminf conj}}\Big)
\label{K=f^la} 
\end{align}
where $f_\la\in\tLipxy$ is a $\la$-weak Kantorovich potential of $\wKD(x,y)$.

On the other hand, by the same proof as that for \autoref{K < C}, we obtain
\begin{align}
d(x,y)\CC(x,y)
&= 
\inf_{f\in\tLipxy}\lim_{\la\dto0}\frac{d(x,y)-\kakko(\Jf,\dex-\dey)}{\la} & \notag \\
&= 
\lim_{\la\dto0}\frac{d(x,y)-\kakko(J_\la g,\dex-\dey)}{\la} &\Big( g\in\tLipxy:\text{ a minimizer} \Big)\notag \\
&\ge \lim_{\la\dto0}\frac{1}{\la}\Bigg( d(x,y)-\sup_{f\in\tLipxy}\kakko(\Jf,\dex-\dey) \Bigg) & \notag \\
&= d(x,y)\kxy. & \label{C>Kxy} 
\end{align}

Hence, we conclude that 
\begin{equation*}
\frac{\Kxy}{d(x,y)}
\stackrel{\text{\eqref{K=f^la}}}{\le}
\kxy
\stackrel{\text{\eqref{C>Kxy}}}{\le}
\mathcal{C}(x,y)
=
\inf_{f\in\Lipxy} \frac{\kakko(\LL^0f,\dex-\dey)}{d(x,y)} 
\end{equation*}
and these inequalities are in fact equality since $f_\la\in\Lipxy$.
\end{proof}

\begin{cor} \label{main corollary}
If \autoref{liminf conj} is true, then $\la$-weak Kantorovich potentials of $\wKD(x,y)$ are minimizers of $\mathcal{C}(x,y)$.
Precisely, we have
\begin{equation*}
\CC(x,y)
=
\frac{\Kxy}{d(x,y)}
=
\frac{\kakko(\LL^0f_\la,\dex-\dey)}{d(x,y)},
\end{equation*}
where $f_\la\in\tLipxy$ is a $\la$-weak Kantorovich potential of $\wKD(x,y)$ for sufficiently small $\la>0$.
\end{cor}

%%%%%%%%%%%%%%%%%%%%%%%%%%%%%%%%%%%%%%%%%%%%%%%%%%%%%%%%%%%%%%%%%%%%%%%%%%%
%%%%%%%%%%%%%%%%%%%%%%%%%%%%%%%%%%%%%%%%%%%%%%%%%%%%%%%%%%%%%%%%%%%%%%%%%%%
\section{Calculations of $\CC(x,y)$ of some hypergraphs} \label{calculation section 4}
%%%%%%%%%%%%%%%%%%%%%%%%%%%%%%%%%%%%%%%%%%%%%%%%%%%%%%%%%%%%%%%%%%%%%%%%%%%
%%%%%%%%%%%%%%%%%%%%%%%%%%%%%%%%%%%%%%%%%%%%%%%%%%%%%%%%%%%%%%%%%%%%%%%%%%%

In this section, we calculate $\CC(x,y)$ for several types of hypergraphs.
Recall that as discussed in \autoref{main theorem}, if \autoref{liminf conj} is true, $\CC(x,y)$ coincides with $\kxy$. 

We first review the LLY curvature and the IKTU curvature of some concrete examples.
We adopt the following notation.

\begin{nota}
Let $\K_0\in\R$.
For a hypergraph $H$, we denote $\K(H)=\K_0$ if $\kxy=\K_0$ holds for any two distinct vertices $x,y$.
The same convention applies to $\kLLY$, $\kIKTU$ and $\CC$.
\end{nota}

First recall the LLY curvature of two simple graphs.

\begin{exa}[{\cite[Examples 1 and 2]{LLY}}] \label{example of graph}
For the unweighted complete graph $K_n$ and the unweighted cycle graph $C_n$ with $n$ vertices, we have
\begin{equation*}
\dis \kLLY(K_n)=\frac{n}{n-1}, \quad
\kLLY(C_n)=
\left\{
\begin{aligned}
&\dis\;3-\frac{n}{2} & \;(n=2,3,4,5,6) , \vspace{1mm} \\
&\; 0 & (n\ge6). \vspace{1mm} 
\end{aligned}
\right. 
\end{equation*}
Moreover, when $G$ is  an unweighted graph, $\kLLY(G)>1$ holds if and only if $G$ is a complete graph.
\end{exa}

As seen in \autoref{example of graph}, $\lim_{n\to\infty}\kLLY(K_n)=1$ and $\lim_{n\to\infty}\kLLY(C_n)=0$ hold.
Thus, it is also expected that the curvature of a characteristic hypergraph converges as the number of vertices diverges.
We will calculate $\CC(x,y)$ of \emph{$1$-regular hypergraphs} (while it seems difficult to calculate that of \emph{cycle hypergraphs}).

\begin{defi}
A hypergraph $H$ is \emph{$1$-regular} if every vertex has only one hyperedge including it.
Notice that $1$-regular hypergraphs are unique up to the difference of weights since our graphs are connected in this paper.
We denote a $1$-regular hypergraph with $n$ vertices by $R_{n,1}$.
\end{defi}

\begin{exa}[{\cite[Examples 6.1 and 6.4]{IKTU}}] \label{IKTU curv example}
For the unweighted $1$-regular hypergraph $R_{3,1}$ and the unweighted complete hypergraph $K\hspace{-0.8mm}H_3$ with $3$ vertices, we have $\kIKTU(R_{3,1})=\kIKTU(K\hspace{-0.8mm}H_3)=3/2$.
\end{exa}

For the limit of unweighted complete hypergraphs, we know the following.

\begin{fact}[{\cite[Example 6.5]{IKTU}}]
For an unweighted complete hypergraph $K\hspace{-0.8mm}H_n$ with $n$ vertices and its vertices $x,y$, $\CC(x,y)\le n/(n-1)$ holds.
Therefore, we find $\lim_{n\to\infty}\kIKTU(K\hspace{-0.8mm}H_n)\le1$.
\end{fact}

On the other hand, we show that the following holds for the limit of $1$-regular hypergraphs.

\begin{prop}[see \autoref{1-regular subsection}] \label{1-regular limit}
\hspace{0mm}
\begin{itemize}
\item[$(1)$] $\dis\lim_{n\to\infty}\kIKTU(R_{n,1})\le\lim_{n\to\infty}\K(R_{n,1})\le0$.
\item[$(2)$] If \autoref{liminf conj} is true, then $\dis\lim_{n\to\infty}\K(R_{n,1})=0$.
\end{itemize}
\end{prop}

For an unweighted hypergraph $H$ and its vertices $x,y$, it seems reasonable to expect that heat can move from $x$ to $y$ most easily when $H$ is complete, while it becomes harder as the number of vertices increases when $H$ is $1$-regular, i.e. $\lim_{n\to\infty}\kIKTU(K\hspace{-0.8mm}H_n)>\lim_{n\to\infty}\kIKTU(R_{n,1})$.
Moreover, when $H$ is $1$-regular, it is unlikely that $\kIKTU(R_{n,1})<0$ even if the number of vertices is very large, and thus we expect $\lim_{n\to\infty}\kIKTU(R_{n,1})=0$ to hold.

In the rest of this paper, we focus on $\CC(x,y)$ of the following hypergraphs:
\begin{itemize}
\item $1$-regular hypergraphs (\autoref{1-regular subsection}),
\item Hypergraphs satisfying $e_V\in E$, where $e_V$ is a hyperedge including all vertices, with $2$ hyperedges (\autoref{|E|=2 subsection}).
\end{itemize}

%%%%%%%%%%%%%%%%%%%%%%%%%%%%%%%%%%%%%%%%%%%%%%%%%%%%%%%%%%%%%%%%%%%%%%%%%%%
\subsection{Preparations for calculation} \label{tejun}
%%%%%%%%%%%%%%%%%%%%%%%%%%%%%%%%%%%%%%%%%%%%%%%%%%%%%%%%%%%%%%%%%%%%%%%%%%%

We often use the following properties in our calculations.

\begin{lem} \label{grad lemma}
Take $e\in E$ and $f\in\Lip$ satisfying $\max_{\mathsf{b}\in\mathsf{B}_e}\kakko(f,\mathsf{b})>0$.
For $\mathsf{b}\in\mathsf{B}_e$, we define $\mathsf{b}^+,\mathsf{b}^-\subset V$ as
\begin{equation*}
\mathsf{b}^+
\:=
\left\{v\in V\;\middle|\;
\mathsf{b}(v)>0\right\},
\quad
\mathsf{b}^-
\:=
\left\{v\in V\;\middle|\;
\mathsf{b}(v)<0\right\},
\end{equation*}
respectively.
Then, for any $\flow(e)\in\mathrm{arg\,max}_{\mathsf{b}\in\mathsf{B}_e}\kakko(f,\mathsf{b})$, we have
\begin{equation}
\sum_{v\in\posiflow(e)}\flow(e)(v)=1
,\quad
\sum_{v\in\negaflow(e)}\flow(e)(v)=-1. \label{flow formula}
\end{equation}
In particular, if $\max_{\mathsf{b}\in\mathsf{B}_e}\kakko(f,\mathsf{b})=1$, then we have
\begin{equation}
\sum_{v\in\posiflow(e)}w_e\kakko(f,{\flow(e)})\flow(e)(v)=w_e
,\quad
\sum_{v\in\negaflow(e)}w_e\kakko(f,{\flow(e)})\flow(e)(v)=-w_e. \label{L0 formula}
\end{equation}
\end{lem}

\begin{proof}
We represent $\flow(e)\in\mathrm{arg\,max}_{\mathsf{b}\in\mathsf{B}_e}\kakko(f,\mathsf{b})$ by 
$\flow(e)=\sum_i a_i(\delta_{p_i}-\delta_{q_i})$ where $p_i,q_i\in V$, $a_i>0$ for any $i$ and $\sum_i a_i=1$.
Then, we have
\begin{equation*}
\kakko(f,{\flow(e)})
=
\sum_i a_i\Big(\kakko(f,\delta_{p_i})-\kakko(f,\delta_{q_i})\Big)
\le
\bigg(\max_{v\in e}\kakko(f,\dev)-\min_{v\in e}\kakko(f,\dev)\bigg)\sum_i a_i
=
\max_{\mathsf{b}\in\mathsf{B}_e}\kakko(f,\mathsf{b}).
\end{equation*}
Therefore, by $\flow(e)\in\mathrm{arg\,max}_{\mathsf{b}\in\mathsf{B}_e}\kakko(f,\mathsf{b})$, we obtain $p_i\in\mathrm{arg\,max}_{v\in e}\kakko(f,\dev)$ and $q_i\in\mathrm{arg\,min}_{v\in e}\kakko(f,\dev)$.
Hence, \eqref{flow formula} holds.
\end{proof}

\begin{thm} \label{key property}
Let $E=\{e_V,e\}$, where $e_V$ is a hyperedge including all vertices.
In addition, let $x,y,z\in V$ and $f\in\Lipxy$.
\begin{itemize}
\item[$(1)$] If $\kakko(f,\dez)=\kakko(f,\dex)$ and $E_z=E_x$, then $\kakko(\LL^0f,\dez)=\kakko(\LL^0f,\dex)$.
\item[$(2)$] If $\kakko(f,\dez)=\kakko(f,\dey)$ and $E_z=E_y$, then $\kakko(\LL^0f,\dez)=\kakko(\LL^0f,\dey)$.
\item[$(3)$] If $\kakko(f,\dez)=\kakko(f,\dex)$ and $f$ is a minimizer of $\CC(x,y)$, then $\kakko(\LL^0f,\dez)=\kakko(\LL^0f,\dex)$.
\item[$(4)$] If $\kakko(f,\dez)=\kakko(f,\dey)$ and $f$ is a minimizer of $\CC(x,y)$, then $\kakko(\LL^0f,\dez)=\kakko(\LL^0f,\dey)$.
\end{itemize}
\end{thm}

\begin{proof} 
Notice the following:
\begin{itemize}
\item
Since $e_V$ includes all vertices and $f\in\Lipxy$, we have
\begin{equation}
\kakko(f,\dex)=\max_{v\in V}\kakko(f,\dev)
\quad\text{ and }\quad
\kakko(f,\dey)=\min_{v\in V}\kakko(f,\dev). \label{max:x min:y}
\end{equation}
\item
By the definition of $\LL$, for any $f^\prime\in\LL f$, there exists some $\mathsf{b}_1\in\mathrm{arg\,max}_{\mathsf{b}\in\mathsf{B}_{e_V}}\kakko(f,\mathsf{b})$ and $\mathsf{b}_2\in\mathrm{arg\,max}_{\mathsf{b}\in\mathsf{B}_e}\kakko(f,\mathsf{b})$ such that $f^\prime = w_{e_V}\mathsf{b}_1+w_e\kakko(f,\mathsf{b}_2)\mathsf{b}_2$.
\end{itemize}
In the following discussion, we denote $\LL^0f=w_{e_V}\mathsf{b}_{e_V}(f)+w_e\kakko(f,{\mathsf{b}_e(f)})\mathsf{b}_e(f)$, where $\mathsf{b}_{e_V}(f)\in\mathrm{arg\,max}_{\mathsf{b}\in\mathsf{B}_{e_V}}\kakko(f,\mathsf{b})$ and $\mathsf{b}_e(f)\in\mathrm{arg\,max}_{\mathsf{b}\in\mathsf{B}_e}\kakko(f,\mathsf{b})$.
\begin{itemize}
\item[$(1)$]
Assuming $\kakko(\LL^0f,\dez)\neq\kakko(\LL^0f,\dex)$, we derive a contradiction to the definition of $\LL^0$ by showing that there exists $\ttilde{\,f\,}\in\LL f$ such that 
\begin{equation}
\big\|\ttilde{\,f\,}\big\|_{\Di}
<
\|\LL^0f\|_{\Di}. \label{mokuhyou}
\end{equation}
We define $\ttilde{\mathsf{b}_{e_V}(f)}\in\mathrm{arg\,max}_{\mathsf{b}\in\mathsf{B}_{e_V}}\kakko(f,\mathsf{b})$ 
and $\ttilde{\mathsf{b}_e(f)}\in\mathrm{arg\,max}_{\mathsf{b}\in\mathsf{B}_e}\kakko(f,\mathsf{b})$ by
\begin{align*}
\dis \ttilde{\mathsf{b}_{e_V}(f)}(v)
&\:=
\left\{
\begin{aligned}
&\;\frac{\;\mathsf{b}_{e_V}(f)(x)+\mathsf{b}_{e_V}(f)(z)\;}{2} & (v=x,z) , \vspace{1mm} \\
&\;\mathsf{b}_{e_V}(f)(v) & (\text{otherwise}) ,
\end{aligned}
\right. \\
\dis \ttilde{\mathsf{b}_e(f)}(v)
&\:=
\left\{
\begin{aligned}
&\;\frac{\;\mathsf{b}_e(f)(x)+\mathsf{b}_e(f)(z)\;}{2} & (v=x,z) , \vspace{1mm} \\
&\;\mathsf{b}_e(f)(v) & (\text{otherwise}),
\end{aligned}
\right. 
\end{align*}
respectively.
Then, $\ttilde{\,f\,}\:=w_{e_V}\ttilde{\mathsf{b}_{e_V}(f)}+w_e\kakko(f,\mathsf{b}_e)\ttilde{\mathsf{b}_e(f)}$ satisfies $\ttilde{\,f\,}\in\mathcal{L}f$ by the hypotheses and
\begin{align*}
\dis \kakko(\ttilde{\,f\,},\dev)=
\left\{
\begin{aligned}
&\;\frac{\kakko(\LL^0f,\dex)+\kakko(\LL^0f,\dez)}{2} & (v=x,z), \vspace{1mm} \\
&\;\kakko(\LL^0f,\dev) & (\text{otherwise}).
\end{aligned}
\right. 
\end{align*}
Therefore, the inequality \eqref{mokuhyou} holds.
Indeed, $\|\LL^0f\|_{\Di}^2>\big\|\ttilde{\,f\,}\big\|_{\Di}^2$ holds since we have $\kakko(\LL^0f,\dez)\neq\kakko(\LL^0f,\dex)$ and $a^2+b^2>2\{(a+b)/2\}^2$ for any real numbers $a\neq b$.

\item[$(2)$]
This follows from the same proof as that for $(1)$.

\item[$(3)$]
Since we proved the case of $E_z=E_x$ in (1), we consider the case of $E_z\neq E_x$, in other words, either $x\in e$ and $z\not\in e$ or $x\not\in e$ and $z\in e$.

\vspace{1mm}
\textbf{\underline{The case $z\in e$ and $x\not\in e$}:} 
We set
\begin{equation*}
\A\:=\min_{v\in e}\kakko(f,\dev)
,\quad
\B\:=\kakko(f,\mathsf{b}_e)=\kakko(f,\dez)-\A\quad\big(\in[0,1]\big).
\end{equation*}
By \eqref{max:x min:y}, it is sufficient to consider the following cases:
\begin{itemize}
\item[\textbf{(A)}] $y\in e$.
\item[\textbf{(B)}] $y\not\in e$ and $\A=\kakko(f,\dez)$.
\item[\textbf{(C)}] $y\not\in e$ and $\A\in\Big(\kakko(f,\dey),\,\kakko(f,\dez)\Big)$.
\item[\textbf{(D)}] $y\not\in e$ and $\A=\kakko(f,\dey)$.
\end{itemize}
\textbf{(A)} \underline{The case $y\in e$}:
\begin{itemize}
\item[$\bullet$]
\textbf{Step 1.}
To illustrate the behavior of heat transfer, it is sufficient to consider the following sets:
\begin{align*}
\Px &\:=
\Big\{v\in V\;\Big|\;\kakko(f,\dev)=\kakko(f,\dex),\;v\not\in e\Big\}\quad(\ni x), \\
\QQ\;\; &\:=
\Big\{v\in V\;\Big|\;\kakko(f,\dev)=\kakko(f,\dey),\;v\not\in e\Big\}, \\
\Rz &\:=
\Big\{v\in V\;\Big|\;\kakko(f,\dev)=\kakko(f,\dez),\;v\in e\Big\}\quad(\ni z), \\
\Sy &\:=
\Big\{v\in V\;\Big|\;\kakko(f,\dev)=\kakko(f,\dey),\;v\in e\Big\}\quad(\ni y).
\end{align*}
By the assumption $\kakko(f,\dez)=\kakko(f,\dex)$ and $(1)$ and $(2)$ above, $\mathsf{b}_{e_V}(f)\in\mathrm{arg\,max}_{\mathsf{b}\in\mathsf{B}_{e_V}}\kakko(f,\mathsf{b})$ 
and $\mathsf{b}_e(f)\in\mathrm{arg\,max}_{\mathsf{b}\in\mathsf{B}_e}\kakko(f,\mathsf{b})$ are represented by
\begin{equation*}
\dis \mathsf{b}_{e_V}(f)(v)
\:=
\left\{
\begin{aligned}
&\;p & (v\in\Px) , \vspace{1mm} \\
&\;-q & (v\in\QQ) , \vspace{1mm} \\
&\;r & (v\in\Rz) , \vspace{1mm} \\
&\;-s & (v\in\Sy) , \vspace{1mm} \\
&\;0 & (\text{otherwise}), 
\end{aligned}
\right. 
\qquad\text{and}\qquad
\mathsf{b}_e(f)(v)
\:=
\left\{
\begin{aligned}
&\;0 & (v\in\Px) , \vspace{1mm} \\
&\;0 & (v\in\QQ) , \vspace{1mm} \\
&\;t & (v\in\Rz) , \vspace{1mm} \\
&\;-u & (v\in\Sy) , \vspace{1mm} \\
&\;0 & (\text{otherwise}) ,
\end{aligned}
\right. 
\end{equation*}
where $p,q,r,s,t,u\in[0,1]$.
Note that we have
\begin{align}
\dis
\left\{
\begin{aligned}
&\;\#\Px\cdot p+\#\Rz\cdot r=1 , \vspace{1mm} \\
&\;\#\QQ\cdot q+\#\Sy\cdot s=1 , \vspace{1mm} 
\end{aligned}
\right. 
\qquad\text{and}\qquad
\left\{
\begin{aligned}
&\;\#\Rz\cdot t=1 , \vspace{1mm} \\
&\;\#\Sy\cdot u=1 . \vspace{1mm} 
\end{aligned}
\right. 
\label{iii-1 condition}
\end{align}
\item[$\bullet$]
\textbf{Step 2.}
We calculate $p$ and $r$.
We have $\B=1$ and
\begin{equation*}
\dis \LL^0f
=
w_{e_V}\mathsf{b}_{e_V}(f)+w_e\mathsf{b}_e(f)
=
\left\{
\begin{aligned}
&\;w_{e_V}\cdot p & (v\in\Px) , \vspace{1mm} \\
&\;-w_{e_V}\cdot q & (v\in\QQ) , \vspace{1mm} \\
&\;w_{e_V}\cdot r+w_e\cdot t & (v\in\Rz) , \vspace{1mm} \\
&\;-w_{e_V}\cdot s-w_e\cdot u & (v\in\Sy) . 
\end{aligned}
\right. 
\end{equation*}
Define 
\begin{equation*}
\|\LL^0f^+\|_{\Di}^2
\:=
\sum_{v\in\Px\cup\Rz}\frac{\LL^0f(v)^2}{d_v}
,\quad
\|\LL^0f^-\|_{\Di}^2
\:=
\sum_{v\in\QQ\cup\Sy}\frac{\LL^0f(v)^2}{d_v}.
\end{equation*}
Since $r=\big(1-\#\Px\cdot p\big)\big/\#\Rz$ and $t=1\big/\#\Rz$ by \eqref{iii-1 condition}, we have
\begin{align*}
\|\LL^0f^+\|_{\Di}^2
&=
\sum_{v\in\Px}w_{e_V}\cdot p^2+\sum_{v\in\Rz}\frac{\dis\;\big(w_{e_V}\cdot r+w_e\cdot t\big)^2\;}{w_{e_V}+w_e} \\
&=
\#\Px\cdot w_{e_V}\cdot p^2+\frac{\#\Rz}{\;w_{e_V}+w_e\;}\Bigg(w_{e_V}\cdot\frac{\;1-\#\Px\cdot p\;}{\#\Rz}+w_e\cdot\frac{1}{\;\#\Rz\;}\Bigg)^2 \\
&=
\#\Px\cdot w_{e_V}\cdot p^2+\frac{1}{\;\#\Rz\;}\cdot\frac{1}{\;w_{e_V}+w_e\;}\Big\{-\#\Px\cdot w_{e_V}\cdot p+(w_{e_V}+w_e)\Big\}^2.
\end{align*}
Hence, we obtain
\begin{align}
\frac{\;\dd\|\LL^0f^+\|_{\Di}^2\;}{\dd p}
&=
2\#\Px\cdot w_{e_V}\cdot p+\frac{2}{\;\#\Rz\;}\cdot\frac{1}{\;w_{e_V}+w_e\;}\Big\{-\#\Px\cdot w_{e_V}\cdot p+(w_{e_V}+w_e)\Big\}\big(-\#\Px\cdot w_{e_V}\big) \notag \\
&=
2\#\Px\cdot w_{e_V}\Bigg(p-\frac{1}{\;\#\Rz\;}\cdot\frac{1}{\;w_{e_V}+w_e\;}\Big\{-\#\Px\cdot w_{e_V}\cdot p+(w_{e_V}+w_e)\Big\}\Bigg) \notag \\
&=
2\#\Px\cdot w_{e_V}\left\{\Bigg(1+\frac{\;\#\Px\;}{\;\#\Rz\;}\cdot\frac{w_{e_V}}{\;w_{e_V}+w_e\;}\Bigg)p-\frac{1}{\;\#\Rz\;}\right\} \notag \\
&=
\frac{2\#\Px\cdot w_{e_V}}{\;\#\Rz(w_{e_V}+w_e)\;}
\Big(\big\{\#\Rz(w_{e_V}+w_e)+\#\Px\cdot w_{e_V}\big\}p-(w_{e_V}+w_e)\Big) \notag \\
&=
\frac{\dis\;2\#\Px\cdot w_{e_V}\big\{\big(\#\Px+\#\Rz\big)w_{e_V}+\#\Rz\cdot w_e\big\}\;}{\;\#\Rz(w_{e_V}+w_e)\;}\Bigg(p-\frac{w_{e_V}+w_e}{\dis\;\big(\#\Px+\#\Rz\big)w_{e_V}+\#\Rz\cdot w_e\;}\Bigg). \label{iii-3 condition}
\end{align}
By \eqref{iii-1 condition}, we find $p\in\big[0,1\big/\#\Px\big]$.
Now, we consider the following cases separately:
\begin{align}
\bullet\;\;
\frac{w_{e_V}+w_e}{\dis\;\big(\#\Px+\#\Rz\big)w_{e_V}+\#\Rz\cdot w_e\;}
\le
\frac{1}{\;\#\Px\;}
&\quad\Big(\Leftrightarrow\;
\#\Px\cdot w_e\le\#\Rz\cdot w_{e_V}+\#\Rz\cdot w_e\Big), \label{iii-3-1 condition} \\
\bullet\;\;
\frac{w_{e_V}+w_e}{\dis\;\big(\#\Px+\#\Rz\big)w_{e_V}+\#\Rz\cdot w_e\;}
>
\frac{1}{\;\#\Px\;}
&\quad\Big(\Leftrightarrow\;
\#\Px\cdot w_e>\#\Rz\cdot w_{e_V}+\#\Rz\cdot w_e\Big).
\label{iii-3-2 condition}
\end{align}
\item[$\bullet$]
\textbf{Step 3.} 
We show that if \eqref{iii-3-1 condition} holds, then $\kakko(\LL^0f,\dez)=\kakko(\LL^0f,\dex)$.
By \eqref{iii-3 condition} and \eqref{iii-3-1 condition}, 
\begin{equation*}
p
=
\frac{w_{e_V}+w_e}{\dis\;\big(\#\Px+\#\Rz\big)w_{e_V}+\#\Rz\cdot w_e\;}\quad\Bigg(\in\bigg[0,\frac{1}{\#\Px}\bigg]\Bigg).
\end{equation*}
Then, we obtain
\begin{equation*}
r
=
\frac{1}{\;\#\Rz\;}\left\{1-\#\Px\cdot\frac{w_{e_V}+w_e}{\dis\;\big(\#\Px+\#\Rz\big)w_{e_V}+\#\Rz\cdot w_e\;}\right\}
=
\frac{1}{\;\#\Rz\;}\cdot\frac{\dis\#\Rz\cdot w_{e_V}+\big(\#\Rz-\#\Px\big)w_e}{\dis\;\big(\#\Px+\#\Rz\big)w_{e_V}+\#\Rz\cdot w_e\;}
\end{equation*}
and $r\in\big[0,1\big/\#\Rz\big]$ by \eqref{iii-1 condition}.
Hence, we obtain $\kakko(\LL^0f,\dez)=\kakko(\LL^0f,\dex)$ since
\begin{align*}
\kakko(\LL^0f,\dez)
&=
\frac{\;w_{e_V}\cdot r+w_e\cdot t\;}{w_{e_V}+w_e}
=
\frac{1}{\;w_{e_V}+w_e\;}\Bigg\{w_{e_V}\cdot\frac{1}{\;\#\Rz\;}\cdot\frac{\dis\#\Rz\cdot w_{e_V}+\big(\#\Rz-\#\Px\big)w_e}{\dis\;\big(\#\Px+\#\Rz\big)w_{e_V}+\#\Rz\cdot w_e\;}+w_e\cdot\frac{1}{\;\#\Rz\;}\Bigg\} \\
&=
\frac{1}{\;\#\Rz\;}\cdot\frac{1}{\;w_{e_V}+w_e\;}\cdot\frac{\dis\;\big\{\#\Rz\cdot w_{e_V}+\big(\#\Rz-\#\Px\big)w_e\big\}w_{e_V}+\big\{\big(\#\Px+\#\Rz\big)w_{e_V}+\#\Rz\cdot w_e\big\}w_e\;}{\dis\;\big(\#\Px+\#\Rz\big)w_{e_V}+\#\Rz\cdot w_e\;} \\
&=
\frac{1}{\;w_{e_V}+w_e\;}\cdot\frac{\dis w_{e_V}^2+2w_{e_V}w_e+w_e^2}{\dis\;\big(\#\Px+\#\Rz\big)w_{e_V}+\#\Rz\cdot w_e\;}
=
\frac{w_{e_V}+w_e}{\dis\;\big(\#\Px+\#\Rz\big)w_{e_V}+\#\Rz\cdot w_e\;} \\
&=
p 
=
\kakko(\LL^0f,\dex).
\end{align*}
\item[$\bullet$]
\textbf{Step 4.} 
We show that if $f$ is a minimizer of $\CC(x,y)$, then \eqref{iii-3-2 condition} cannot hold.
Precisely, assuming \eqref{iii-3-2 condition}, we shall show that there exists $g\in\Lipxy$ such that
\begin{equation}
\kakko(\LL^0g,\dex-\dey) < \kakko(\LL^0f,\dex-\dey). \label{iii-3-2 goal}
\end{equation}
\begin{itemize}
\item
First, we calculate $\kakko(\LL^0f,\dex-\dey)$.
In the case \eqref{iii-3-2 condition}, since $\dd\|\LL^0f^+\|_{\Di}^2\big/\dd p<0$ by \eqref{iii-3 condition}, we find $(p,r)=\big(1\big/\#\Px,0\big)$.
\begin{itemize}
\item[--]
If $\QQ=\emptyset$, then we find $(q,s)=\big(0,1\big/\#\Sy\big)$.
Moreover, $u=1\big/\#\Sy$ holds by \eqref{iii-1 condition}.
Thus, we have
\begin{equation}
\kakko(\LL^0f,\dex-\dey)
=
p+\frac{\;w_{e_V}\cdot s+w_e\cdot u\;}{w_{e_V}+w_e}
=
\frac{1}{\dis\;\#\Px\;}+\frac{1}{\dis\;\#\Sy\;}. \label{q condition emptyset}
\end{equation}
\item[--]
If $\QQ\neq\emptyset$, then a similar calculation to \textbf{Step 2} yields 
\begin{equation}
\frac{\;\dd\|\LL^0f^-\|_{\Di}^2\;}{\dd q}
=
\frac{\dis\;2\#\QQ\cdot w_{e_V}\big\{\big(\#\QQ+\#\Sy\big)w_{e_V}+\#\Sy\cdot w_e\big\}\;}{\;\#\Sy(w_{e_V}+w_e)\;}\Bigg(q-\frac{w_{e_V}+w_e}{\dis\;\big(\#\QQ+\#\Sy\big)w_{e_V}+\#\Sy\cdot w_e\;}\Bigg). \label{q condition}
\end{equation}
By \eqref{iii-1 condition}, similarly to \eqref{iii-3-1 condition} and \eqref{iii-3-2 condition}, we consider the following cases:
\begin{align}
\bullet\;\;
\frac{w_{e_V}+w_e}{\dis\;\big(\#\QQ+\#\Sy\big)w_{e_V}+\#\Sy\cdot w_e\;}
\le
\frac{1}{\;\#\QQ\;}
&\quad\Big(\Leftrightarrow\;
\#\QQ\cdot w_e\le\#\Sy\cdot w_{e_V}+\#\Sy\cdot w_e\Big), \label{iii-3-3 condition} \\
\bullet\;\;
\frac{w_{e_V}+w_e}{\dis\;\big(\#\QQ+\#\Sy\big)w_{e_V}+\#\Sy\cdot w_e\;}
>
\frac{1}{\;\#\QQ\;}
&\quad\Big(\Leftrightarrow\;
\#\QQ\cdot w_e>\#\Sy\cdot w_{e_V}+\#\Sy\cdot w_e\Big).
\label{iii-3-4 condition}
\end{align}
\underline{In the case \eqref{iii-3-3 condition}}, we observe
\begin{equation*}
q 
=
\frac{w_{e_V}+w_e}{\dis\;\big(\#\QQ+\#\Sy\big)w_{e_V}+\#\Sy\cdot w_e\;}
,\quad
s
=
\frac{1}{\;\#\Sy\;}\cdot\frac{\;\#\Sy\cdot w_{e_V}+\big(\#\Sy-\#\QQ\big)w_e\;}{\;\big(\#\QQ+\#\Sy\big)w_{e_V}+\#\Sy\cdot w_e\;}
,\quad
u
=
\frac{1}{\;\#\Sy\;}
\end{equation*}
by \eqref{q condition} and \eqref{iii-1 condition}.
Therefore, we obtain
\begin{equation}
\kakko(\LL^0f,\dex-\dey)
=
p+\frac{\;w_{e_V}\cdot s+w_e\cdot u\;}{w_{e_V}+w_e}
=
\frac{1}{\dis\;\#\Px\;}+\frac{w_{e_V}+w_e}{\;\big(\#\QQ+\#\Sy\big)w_{e_V}+\#\Sy\cdot w_e\;}. \label{fxy 1-1}
\end{equation}
\underline{In the case \eqref{iii-3-4 condition}}, since $\dd\|\LL^0f^+\|_{\Di}^2\big/\dd q<0$ holds by \eqref{q condition}, we have $(q,s)=\big(1\big/\#\QQ,0\big)$.
Therefore, we obtain
\begin{equation}
\kakko(\LL^0f,\dex-\dey)
=
p+\frac{\;w_e\cdot u\;}{\;w_{e_V}+w_e\;}
=
\frac{1}{\dis\;\#\Px\;}+\frac{1}{\dis\;\#\Sy\;}\cdot\frac{w_e}{\;w_{e_V}+w_e\;}. \label{fxy 1-2}
\end{equation}
\end{itemize}
\item
Next, we construct $g\in\Lipxy$ satisfying \eqref{iii-3-2 goal}.
We define $g\in\Lipxy$ by
\begin{equation*}
\dis g(v)
=
\left\{
\begin{aligned}
&\;f(y) & (v\in e) , \vspace{1mm} \\
&\;f(v) & (v\not\in e) . 
\end{aligned}
\right. 
\end{equation*}
By the definition of $g$, for any $\mathsf{b}_e(g)\in\mathrm{arg\,max}_{\mathsf{b}\in\mathsf{B}_e}\kakko(g,\mathsf{b})$ and $v\in V$, we have
\begin{equation}
\kakko(g,{\mathsf{b}_e(g)})\mathsf{b}_e(g)(v)=0. \label{no flow on e}
\end{equation}
By $(1)$ and $(2)$, $\mathsf{b}_{e_V}(g)\in\mathrm{arg\,max}_{\mathsf{b}\in\mathsf{B}_{e_V}}\kakko(g,\mathsf{b})$ is represented by
\begin{equation*}
\dis \mathsf{b}_{e_V}(g)(v)
\:=
\left\{
\begin{aligned}
&\;p^\prime & (v\in\Px) , \vspace{1mm} \\
&\;-q^\prime & (v\in\QQ) , \vspace{1mm} \\
&\;-s^\prime & (v\in e) , 
\end{aligned}
\right. 
\end{equation*}
where $p^\prime,q^\prime,s^\prime\in[0,1]$.
Therefore, we find $p^\prime=1\big/\#\Px$.
\begin{itemize}
\item[--]
If $\QQ=\emptyset$, then we have $(q^\prime,s^\prime)=\big(0,1\big/\#e\big)$.
Thus, we obtain
\begin{equation}
\kakko(\LL^0g,\dex-\dey)
=
p^\prime+\frac{w_{e_V}\cdot s^\prime}{\;w_{e_V}+w_e\;}
=
\frac{1}{\dis\;\#\Px\;}+\frac{w_{e_V}}{\;\#e(w_{e_V}+w_e)\;}. \label{gxy 1-0}
\end{equation}
\item[--]
If $\QQ\neq\emptyset$, then we have
\begin{equation*}
q^\prime
=
\frac{w_{e_V}}{\dis\;\big(\#\QQ+\#e\big)w_{e_V}+\#e\cdot w_e\;}\quad\Bigg(\in\bigg[0,\frac{1}{\#\QQ}\bigg]\Bigg)
,\quad
s^\prime
=
\frac{w_{e_V}+w_e}{\dis\;\big(\#\QQ+\#e\big)w_{e_V}+\#e\cdot w_e\;}\quad\Bigg(\in\bigg[0,\frac{1}{\#e}\bigg]\Bigg)
\end{equation*}
in the same way as \textbf{Step 2}.
Thus, we obtain
\begin{equation}
\kakko(\LL^0g,\dex-\dey)
=
p^\prime+\frac{w_{e_V}\cdot s^\prime}{\;w_{e_V}+w_e\;}
=
\frac{1}{\dis\;\#\Px\;}+\frac{w_{e_V}}{\;\big(\#\QQ+\#e\big)w_{e_V}+\#e\cdot w_e\;}. \label{gxy 1-1}
\end{equation}
\end{itemize}
\item
Finally, we show that \eqref{iii-3-2 goal} holds.
\begin{itemize}
\item[--]
If $\QQ=\emptyset$, then $\eqref{q condition emptyset}>\eqref{gxy 1-0}$ holds.
\item[--]
If $\QQ\neq\emptyset$ and \eqref{iii-3-3 condition} holds, then $\eqref{fxy 1-1}>\eqref{gxy 1-1}$ holds by $\#\Sy<\#e$ since $z\in e$.
\item[--]
If $\QQ\neq\emptyset$ and \eqref{iii-3-4 condition} holds, then $\eqref{fxy 1-2}>\eqref{gxy 1-1}$ holds.
Indeed, we have
\begin{align*}
&\;\quad 
w_e\Big\{\big(\#\QQ+\#e\big)w_{e_V}+\#e\cdot w_e\Big\}-w_{e_V}\cdot\#\Sy(w_{e_V}+w_e) \\
&=
\big(\#\QQ+\#e-\#\Sy\big)w_{e_V}w_e+\#e\cdot w_e^2-\#\Sy\cdot w_{e_V}^2 \\
&\ge
\big(\#\QQ+\#\Rz\big)w_{e_V}w_e+\#e\cdot w_e^2-\#\Sy\cdot w_{e_V}^2 \\
&>
\#\Sy(w_{e_V}+w_e)w_{e_V}+\#\Rz\cdot w_{e_V}w_e+\#e\cdot w_e^2-\#\Sy\cdot w_{e_V}^2 &\Big(\text{by \eqref{iii-3-4 condition}}\Big) \\
&=
\big(\#\Sy+\#\Rz\big)w_{e_V}w_e+\#e\cdot w_e^2 
>
0.
\end{align*}
\end{itemize}
Therefore, we obtain \eqref{iii-3-2 goal}.
\end{itemize}

This completes the proof of the case of \textbf{(A)}.
$\blacksquare$
\end{itemize}

\textbf{(B)} \underline{The case $y\not\in e$ and $\A=\kakko(f,\dez)$}:
We discuss similarly to \textbf{(A)}.
By $\A=\kakko(f,\dez)=\kakko(f,\dex)$, for any $\mathsf{b}_e(f)\in\mathrm{arg\,max}_{\mathsf{b}\in\mathsf{B}_e}\kakko(f,\mathsf{b})$ and $v\in V$, we have $\B=0$.
Hence, we consider
\begin{align*}
\Px &\:=
\Big\{v\in V\;\Big|\;\kakko(f,\dev)=\kakko(f,\dex),\;v\not\in e\Big\}\quad(\ni x), \\
\Qy &\:=
\Big\{v\in V\;\Big|\;\kakko(f,\dev)=\kakko(f,\dey),\;v\not\in e\Big\}\quad(\ni y),
\end{align*}
and $\mathsf{b}_{e_V}(f)\in\mathrm{arg\,max}_{\mathsf{b}\in\mathsf{B}_{e_V}}\kakko(f,\mathsf{b})$ represented by
\begin{equation*}
\dis \mathsf{b}_{e_V}(f)(v)
\:=
\left\{
\begin{aligned}
&\;p & (v\in\Px) , \vspace{1mm} \\
&\;-q & (v\in\Qy) , \vspace{1mm} \\
&\;r & (v\in e), 
\end{aligned}
\right. 
\end{equation*}
where $p,q,r\in[0,1]$.
Thus, we have $q=1\big/\#\Qy$.
Moreover, as in \textbf{(A)}, we observe
\begin{equation*}
p
=
\frac{w_{e_V}}{\dis\;\big(\#\Px+\#e\big)w_{e_V}+\#e\cdot w_e\;}\quad\Bigg(\in\bigg[0,\frac{1}{\#\Px}\bigg]\Bigg)
,\quad
r
=
\frac{w_{e_V}+w_e}{\dis\;\big(\#\Px+\#e\big)w_{e_V}+\#e\cdot w_e\;}\quad\Bigg(\in\bigg[0,\frac{1}{\#e}\bigg]\Bigg).
\end{equation*}
This implies that
\begin{equation*}
\kakko(\LL^0f,\dez)
=
\frac{w_{e_V}\cdot r}{\;w_{e_V}+w_e\;}
=
\frac{w_{e_V}}{\dis\;\big(\#\Px+\#e\big)w_{e_V}+\#e\cdot w_e\;}
=
p
=
\kakko(\LL^0f,\dex).
\end{equation*}
This completes the proof of the case of \textbf{(B)}.
$\blacksquare$

\textbf{(C)} \underline{The case $y\not\in e$ and $\A\in\Big(\kakko(f,\dey),\,\kakko(f,\dez)\Big)$}:
In this case, we need to consider
\begin{align*}
\Px &\:=
\Big\{v\in V\;\Big|\;\kakko(f,\dev)=\kakko(f,\dex),\;v\not\in e\Big\}\quad(\ni x), \\
\Qy &\:=
\Big\{v\in V\;\Big|\;\kakko(f,\dev)=\kakko(f,\dey),\;v\not\in e\Big\}\quad(\ni y), \\
\Rz &\:=
\Big\{v\in V\;\Big|\;\kakko(f,\dev)=\kakko(f,\dez),\;v\in e\Big\}\quad(\ni z), \\
\SS\;\; &\:=
\Big\{v\in V\;\Big|\;\kakko(f,\dev)=\A,\;v\in e\Big\}.
\end{align*}
Then, $\mathsf{b}_{e_V}(f)\in\mathrm{arg\,max}_{\mathsf{b}\in\mathsf{B}_{e_V}}\kakko(f,\mathsf{b})$ 
and $\mathsf{b}_e(f)\in\mathrm{arg\,max}_{\mathsf{b}\in\mathsf{B}_e}\kakko(f,\mathsf{b})$ are represented by
\begin{equation*}
\dis \mathsf{b}_{e_V}(f)(v)
\:=
\left\{
\begin{aligned}
&\;p & (v\in\Px) , \vspace{1mm} \\
&\;-q & (v\in\Qy) , \vspace{1mm} \\
&\;r & (v\in\Rz) , \vspace{1mm} \\
&\;0 & (v\in\SS) , \vspace{1mm} \\
&\;0 & (\text{otherwise}), 
\end{aligned}
\right. 
\qquad\text{and}\qquad
\mathsf{b}_e(f)(v)
\:=
\left\{
\begin{aligned}
&\;0 & (v\in\Px) , \vspace{1mm} \\
&\;0 & (v\in\Qy) , \vspace{1mm} \\
&\;t & (v\in\Rz) , \vspace{1mm} \\
&\;-u & (v\in\SS) , \vspace{1mm} \\
&\;0 & (\text{otherwise}) ,
\end{aligned}
\right. 
\end{equation*}
where $p,q,r,t,u\in[0,1]$.
Note that $q=1\big/\#\Qy$, $t=1\big/\#\Rz$, $u=1\big/\#\SS$ and $\B=\kakko(f,\dez)-\A\in(0,1)$.
Then, since we have $\LL^0f=w_{e_V}\mathsf{b}_{e_V}(f)+\B w_e\mathsf{b}_e(f)$, by a similar calculation to \textbf{(A)}, we obtain
\begin{align}
\frac{\;\dd\|\LL^0f^+\|_{\Di}^2\;}{\dd p}
&=
\frac{\dd}{\;\dd p\;}\sum_{v\in\Px\cup\Rz}\frac{\;\LL^0f(v)^2\;}{d_v} \notag \\
&=
\frac{\dis\;2\#\Px\cdot w_{e_V}\big\{\big(\#\Px+\#\Rz\big)w_{e_V}+\#\Rz\cdot w_e\big\}\;}{\;\#\Rz(w_{e_V}+w_e)\;}\Bigg(p-\frac{w_{e_V}+\B w_e}{\dis\;\big(\#\Px+\#\Rz\big)w_{e_V}+\#\Rz\cdot w_e\;}\Bigg). \label{alpha condition}
\end{align}
Therefore, we consider the following cases:
\begin{align}
\bullet\;\;
\frac{w_{e_V}+\B w_e}{\dis\;\big(\#\Px+\#\Rz\big)w_{e_V}+\#\Rz\cdot w_e\;}
\le
\frac{1}{\;\#\Px\;}
&\quad\Big(\Leftrightarrow\;
\B\#\Px\cdot w_e\le\#\Rz\cdot w_{e_V}+\#\Rz\cdot w_e\Big), \label{iii-alpha-1 condition} \\
\bullet\;\;
\frac{w_{e_V}+\B w_e}{\dis\;\big(\#\Px+\#\Rz\big)w_{e_V}+\#\Rz\cdot w_e\;}
>
\frac{1}{\;\#\Px\;}
&\quad\Big(\Leftrightarrow\;
\B\#\Px\cdot w_e>\#\Rz\cdot w_{e_V}+\#\Rz\cdot w_e\Big).
\label{iii-alpha-2 condition}
\end{align}
We show that if \eqref{iii-alpha-1 condition} or \eqref{iii-alpha-2 condition} holds, then there exists $g\in\Lipxy$ satisfying \eqref{iii-3-2 goal} to conclude that $f$ cannot be a minimizer of $\CC(x,y)$.
\begin{itemize}
\item[$\ast$]
First, we calculate $\kakko(\LL^0f,\dex-\dey)$ in each case.
\begin{itemize}
\item[--]
If \eqref{iii-alpha-1 condition} holds, then by \eqref{alpha condition}, we find
\begin{equation*}
p
=
\frac{w_{e_V}+\B w_e}{\dis\;\big(\#\Px+\#\Rz\big)w_{e_V}+\#\Rz\cdot w_e\;}
,\quad
r
=
\frac{1}{\dis\;\#\Rz\;}\cdot\frac{\dis\;\#\Rz\cdot w_{e_V}+(\#\Rz-\B\#\Px)w_e\;}{\dis\;\big(\#\Px+\#\Rz\big)w_{e_V}+\#\Rz\cdot w_e\;}.
\end{equation*}
Thus, we obtain
\begin{equation}
\kakko(\LL^0f,\dex-\dey)
=
p+q
=
\frac{w_{e_V}+\B w_e}{\dis\;\big(\#\Px+\#\Rz\big)w_{e_V}+\#\Rz\cdot w_e\;} + \frac{1}{\dis\;\#\Qy\;}. \label{alpha x-y 1}
\end{equation}
\item[--]
If \eqref{iii-alpha-2 condition} holds, then by \eqref{alpha condition}, we find $(p,r)=\big(1\big/\#\Px,0\big)$.
Thus, we obtain
\begin{equation}
\kakko(\LL^0f,\dex-\dey)
=
p+q
=
\frac{1}{\dis\;\#\Px\;} + \frac{1}{\dis\;\#\Qy\;}. \label{alpha x-y 2}
\end{equation}
\end{itemize}
\item[$\ast$]
Next, we define $g\in\Lipxy$ by
\begin{equation*}
\dis g(v)
=
\left\{
\begin{aligned}
&\;f(z) & (v\in e) , \vspace{1mm} \\
&\;f(v) & (v\not\in e) . 
\end{aligned}
\right. 
\end{equation*}
By the definition of $g$, for any $\mathsf{b}_e(g)\in\mathrm{arg\,max}_{\mathsf{b}\in\mathsf{B}_e}\kakko(g,\mathsf{b})$ and $v\in V$, we have \eqref{no flow on e}.
Moreover, $\mathsf{b}_{e_V}(g)\in\mathrm{arg\,max}_{\mathsf{b}\in\mathsf{B}_{e_V}}\kakko(g,\mathsf{b})$ is represented by
\begin{equation*}
\dis \mathsf{b}_{e_V}(g)(v)
\:=
\left\{
\begin{aligned}
&\;p^\prime & (v\in\Px) , \vspace{1mm} \\
&\;-q^\prime & (v\in\Qy) , \vspace{1mm} \\
&\;s^\prime & (v\in e) , 
\end{aligned}
\right. 
\end{equation*}
where $p^\prime,q^\prime,s^\prime\in[0,1]$.
Then, by the same calculation as \textbf{Step 2} in \textbf{(B)}, we obtain $q^\prime=1\big/\#\Qy$ and 
\begin{equation*}
p^\prime
=
\frac{w_{e_V}}{\dis\;\big(\#\Px+\#e\big)w_{e_V}+\#e\cdot w_e\;}\quad\Bigg(\in\bigg[0,\frac{1}{\#\Px}\bigg]\Bigg)
,\quad
s^\prime
=
\frac{w_{e_V}+w_e}{\dis\;\big(\#\Px+\#e\big)w_{e_V}+\#e\cdot w_e\;}\quad\Bigg(\in\bigg[0,\frac{1}{\#e}\bigg]\Bigg).
\end{equation*} 
Hence, 
\begin{equation}
\kakko(\LL^0g,\dex-\dey)
=
p^\prime+q^\prime
=
\frac{w_{e_V}}{\dis\;\big(\#\Px+\#e\big)w_{e_V}+\#e\cdot w_e\;}+\frac{1}{\dis\;\#\Qy\;}. \label{alpha x-y}
\end{equation}
\item[$\ast$]
Finally, we show that \eqref{iii-3-2 goal} holds.
\begin{itemize}
\item[--]
In the case \eqref{iii-alpha-1 condition}, we have $\eqref{alpha x-y 1}>\eqref{alpha x-y}$ by $\B>0$ and $\#\Rz<\#e$.
\item[--]
In the case \eqref{iii-alpha-2 condition}, we have $\eqref{alpha x-y 2}>\eqref{alpha x-y}$ by $\#e>0$.
\end{itemize}
Therefore, we obtain \eqref{iii-3-2 goal}.
\end{itemize}
This completes the proof of the case of \textbf{(C)}.
$\blacksquare$

\textbf{(D)} \underline{The case $y\not\in e$ and $\A=\kakko(f,\dey)$}:
We have $\B=1$.
As in \textbf{(C)}, we consider $\Px,\Qy,\Rz,\SS$.
Therefore, as in \textbf{Step 4} of \textbf{(A)}, it is sufficient to show that assuming \eqref{iii-3-2 condition}, there exists $g\in\Lipxy$ satisfying \eqref{iii-3-2 goal}.
\begin{itemize}
\item[$\ast$]
First, we calculate $\kakko(\LL^0f,\dex-\dey)$.
By a similar calculation to \eqref{q condition}, we have
\begin{equation*}
\frac{\;\dd\|\LL^0f^-\|_{\Di}^2\;}{\dd q}
=
\frac{\dis\;2\#\Qy\cdot w_{e_V}\big\{\big(\#\Qy+\#\SS\big)w_{e_V}+\#\SS\cdot w_e\big\}\;}{\;\#\SS(w_{e_V}+w_e)\;}\Bigg(q-\frac{w_{e_V}+w_e}{\dis\;\big(\#\Qy+\#\SS\big)w_{e_V}+\#\SS\cdot w_e\;}\Bigg).
\end{equation*}
Then, we consider the following cases:
\begin{align}
\bullet\;\;
\frac{w_{e_V}+w_e}{\dis\;\big(\#\Qy+\#\SS\big)w_{e_V}+\#\SS\cdot w_e\;}
\le
\frac{1}{\;\#\Qy\;}
&\quad\Big(\Leftrightarrow\;
\#\Qy\cdot w_e\le\#\SS\cdot w_{e_V}+\#\SS\cdot w_e\Big), \label{iv-1 condition} 
\\
\bullet\;\;
\frac{w_{e_V}+w_e}{\dis\;\big(\#\Qy+\#\SS\big)w_{e_V}+\#\SS\cdot w_e\;}
>
\frac{1}{\;\#\Qy\;}
&\quad\Big(\Leftrightarrow\;
\#\Qy\cdot w_e>\#\SS\cdot w_{e_V}+\#\SS\cdot w_e\Big).
\label{iv-2 condition}
\end{align}
\begin{itemize}
\item[--]
In the case \eqref{iv-1 condition}, we observe
\begin{equation*}
q 
=
\frac{w_{e_V}+w_e}{\dis\;\big(\#\Qy+\#\SS\big)w_{e_V}+\#\SS\cdot w_e\;}
,\quad
s
=
\frac{1}{\;\#\SS\;}\cdot\frac{\;\#\SS\cdot w_{e_V}+\big(\#\SS-\#\Qy\big)w_e\;}{\;\big(\#\Qy+\#\SS\big)w_{e_V}+\#\SS\cdot w_e\;}.
\end{equation*}
Therefore, we obtain
\begin{equation}
\kakko(\LL^0f,\dex-\dey)
=
p+q
=
\frac{1}{\dis\;\#\Px\;}+\frac{w_{e_V}+w_e}{\dis\;\big(\#\Qy+\#\SS\big)w_{e_V}+\#\SS\cdot w_e\;}. \label{fxy 4-1}
\end{equation}
\item[--]
In the case \eqref{iv-2 condition}, we find $(q,s)=\big(1\big/\#\Qy,0\big)$.
Therefore, we obtain
\begin{equation}
\kakko(\LL^0f,\dex-\dey)
=
p+q
=
\frac{1}{\dis\;\#\Px\;}+\frac{1}{\dis\;\#\Qy\;}. \label{fxy 4-2}
\end{equation}
\end{itemize}
\item[$\ast$]
Next, we construct $g\in\Lipxy$ satisfying \eqref{iii-3-2 goal}.
We define $g\in\Lipxy$ by
\begin{equation}
\dis g(v)
=
\left\{
\begin{aligned}
&\;f(y) & (v\in e) , \vspace{1mm} \\
&\;f(v) & (v\not\in e) . 
\end{aligned}
\right. \label{g: no flow on e}
\end{equation}
By the definition of $g$, for any $\mathsf{b}_e(g)\in\mathrm{arg\,max}_{\mathsf{b}\in\mathsf{B}_e}\kakko(g,\mathsf{b})$ and $v\in V$, we have \eqref{no flow on e}.
By $(1)$ and $(2)$, $\mathsf{b}_{e_V}(g)\in\mathrm{arg\,max}_{\mathsf{b}\in\mathsf{B}_{e_V}}\kakko(g,\mathsf{b})$ is represented by
\begin{equation*}
\dis \mathsf{b}_{e_V}(g)(v)
\:=
\left\{
\begin{aligned}
&\;p^\prime & (v\in\Px) , \vspace{1mm} \\
&\;-q^\prime & (v\in\Qy) , \vspace{1mm} \\
&\;-s^\prime & (v\in e) , 
\end{aligned}
\right. 
\end{equation*}
where $p^\prime,q^\prime,s^\prime\in[0,1]$.
Therefore, we have $p^\prime=1\big/\#\Px$ and 
\begin{equation*}
q^\prime
=
\frac{w_{e_V}}{\dis\;\big(\#\Qy+\#e\big)w_{e_V}+\#e\cdot w_e\;}\quad\Bigg(\in\bigg[0,\frac{1}{\#\QQ}\bigg]\Bigg)
,\quad
s^\prime
=
\frac{w_{e_V}+w_e}{\dis\;\big(\#\Qy+\#e\big)w_{e_V}+\#e\cdot w_e\;}\quad\Bigg(\in\bigg[0,\frac{1}{\#e}\bigg]\Bigg)
\end{equation*}
in the same way as \textbf{Step 2} in \textbf{(A)}.
Thus, we obtain
\begin{equation}
\kakko(\LL^0g,\dex-\dey)
=
p^\prime+q^\prime
=
\frac{1}{\dis\;\#\Px\;}+\frac{w_{e_V}}{\dis\;\big(\#\Qy+\#e\big)w_{e_V}+\#e\cdot w_e\;}. \label{gxy 4-1}
\end{equation}
\item[$\ast$]
Finally, we show that \eqref{iii-3-2 goal} holds.
\begin{itemize}
\item[--]
In the case \eqref{iv-1 condition}, it holds that $\eqref{fxy 4-1}>\eqref{gxy 4-1}$ by $\#\SS<\#e$ since $z\not\in\SS$.
\item[--]
In the case \eqref{iv-2 condition}, it holds that $\eqref{fxy 4-2}>\eqref{gxy 4-1}$ by $\#e>0$.
\end{itemize}
Therefore, we obtain \eqref{iii-3-2 goal}.
\end{itemize}

This completes the proof of the case of \textbf{(D)}.
$\blacksquare$

\vspace{1mm}
\textbf{\underline{The case $z\not\in e$ and $x\in e$}:} 
This follows from a similar proof to that of the case $z\in e$ and $x\not\in e$.
\item[$(4)$]
This case can be reduced to $(3)$.
Indeed, if $f$ is a minimizer of $\CC(x,y)$, then $-f$ is a minimizer of $\CC(y,x)=\CC(x,y)$.
\end{itemize}
\end{proof}

Although it is unknown under what weight conditions \autoref{key property} holds for general hypergraphs, we expect it to be true for at least unweighted hypergraphs.

\begin{cor} \label{key cor}
Let $E=\{e_V,e\}$, where $e_V$ is a hyperedge including all vertices.
In addition, let $x,y\in V$ and $f\in\Lipxy$.
If $f$ is a minimizer of $\CC(x,y)$, then 
\begin{equation*}
\min_{v\in e}\kakko(f,\dev),\;\max_{v\in e}\kakko(f,\dev)\in\big\{\kakko(f,\dey),\;\kakko(f,\dex)\big\}
\quad\text{ and }\quad
\max_{\mathsf{b}\in\mathsf{B}_e}\kakko(f,\mathsf{b})\in\{0,1\}.
\end{equation*}
\end{cor}

\begin{proof}
In the case $x,y\in e$, we clearly have the assertions with $\max_{\mathsf{b}\in\mathsf{B}_e}\kakko(f,\mathsf{b})=1$.

Assume $x\not\in e$ or $y\not\in e$.
We remark that by the proof of \autoref{key property}(3), if $f$ is a minimizer of $\CC(x,y)$ and there exists $z\in e$ such that $\kakko(f,\dez)=\kakko(f,\dex)$, then $\min_{v\in e}\kakko(f,\dev)\not\in\big(\kakko(f,\dey),\kakko(f,\dex)\big)$, in particular $\min_{v\in e}\kakko(f,\dev)\in\big\{\kakko(f,\dey),\kakko(f,\dex)\big\}$ and $\max_{\mathsf{b}\in\mathsf{B}_e}\kakko(f,\mathsf{b})\in\{0,1\}$.
Similarly, if there exists $z\in e$ such that $\kakko(f,\dez)=\kakko(f,\dey)$, then $\max_{v\in e}\kakko(f,\dev)\in\big\{\kakko(f,\dey),\kakko(f,\dex)\big\}$ and $\max_{\mathsf{b}\in\mathsf{B}_e}\kakko(f,\mathsf{b})\in\{0,1\}$.
Therefore, it is sufficient to prove that assuming $\kakko(f,\dey)<\min_{v\in e}\kakko(f,\dev)\le\max_{v\in e}\kakko(f,\dev)<\kakko(f,\dex)$, there exists $g\in\Lipxy$ such that $\kakko(\LL^0g,\dex-\dey)<\kakko(\LL^0f,\dex-\dey)$.
Note that by assumption and \autoref{key property}(1)(2), we have $x,y\not\in e$ and
\begin{equation*}
\kakko(\LL^0f,\dex-\dey)
=
\frac{1}{\dis\;\#\big\{v\in V\;\big|\;\kakko(f,\dev)=\kakko(f,\dex)\big\}\;}+\frac{1}{\;\#\big\{v\in V\;\big|\;\kakko(f,\dev)=\kakko(f,\dey)\big\}\;}.
\end{equation*}
We define $g\in\Lipxy$ as in \eqref{g: no flow on e}.
Then, by a similar calculation to \eqref{gxy 4-1}, we obtain
\begin{align*}
\kakko(\LL^0g,\dex-\dey)
&=
\frac{1}{\dis\;\#\big\{v\in V\;\big|\;\kakko(g,\dev)=\kakko(g,\dex)\big\}\;}+\frac{w_{e_V}}{\;\Big(\#\big\{v\in V\backslash e\;\big|\;\kakko(g,\dev)=\kakko(g,\dey)\big\}+\#e\Big)w_{e_V}+\#e\cdot w_e\;} \\
&=
\frac{1}{\dis\;\#\big\{v\in V\;\big|\;\kakko(f,\dev)=\kakko(f,\dex)\big\}\;}+\frac{w_{e_V}}{\;\Big(\#\big\{v\in V\;\big|\;\kakko(f,\dev)=\kakko(f,\dey)\big\}+\#e\Big)w_{e_V}+\#e\cdot w_e\;} \\
&<
\frac{1}{\dis\;\#\big\{v\in V\;\big|\;\kakko(f,\dev)=\kakko(f,\dex)\big\}\;}+\frac{1}{\;\#\big\{v\in V\;\big|\;\kakko(f,\dev)=\kakko(f,\dey)\big\}+\#e\;}.
\end{align*}
Hence, $\kakko(\LL^0g,\dex-\dey)<\kakko(\LL^0f,\dex-\dey)$ holds.
\end{proof}

\begin{thm} \label{main theorem 2}
Let $E=\{e_V,e\}$, where $e_V$ is a hyperedge including all vertices.
In addition, let $x,y\in V$ and $f\in\Lipxy$.
If $f$ is a minimizer of $\CC(x,y)$, then $\kakko(f,\dev)=\kakko(f,\dex)$ or $\kakko(f,\dev)=\kakko(f,\dey)$ holds for every $v\in V$.
\end{thm}

\begin{proof}
For $x,y\in V$ and $f\in\Lipxy$, we consider the following sets:
\begin{align*}
\Px(f) &\:=
\Big\{v\in V\;\Big|\;\kakko(f,\dev)=\kakko(f,\dex),\;E_v=E_x\Big\}\quad(\ni x), \\
\Qy(f) &\:=
\Big\{v\in V\;\Big|\;\kakko(f,\dev)=\kakko(f,\dey),\;E_v=E_y\Big\}\quad(\ni y), \\
\Rx(f) &\:=
\Big\{v\in V\;\Big|\;\kakko(f,\dev)=\kakko(f,\dex),\;E_v\neq E_x\Big\}, \\
\Sy(f) &\:=
\Big\{v\in V\;\Big|\;\kakko(f,\dev)=\kakko(f,\dey),\;E_v\neq E_y\Big\}, \\
\VV(f) &\:=
V\big\backslash\big(\Px(f)\cup\Qy(f)\cup\Rx(f)\cup\Sy(f)\big).
\end{align*}
We show that if $\VV(f)\neq\emptyset$, then $f$ is not a minimizer of $\CC(x,y)$.
Assume that $\VV(f)\neq\emptyset$.
Then, we have $\LL^0f(v)=0$ for all $v\in\VV(f)$ by \autoref{key cor}.
Therefore, we observe
\begin{align*}
\|\LL^0f\|_{\Di}^2
&=
\sum_{v\in V}d_v\kakko(\LL^0f,\dev)^2 
=
\sum_{v\in\Px(f)\cup\Rx(f)}d_v\kakko(\LL^0f,\dev)^2+\sum_{v\in\Qy(f)\cup\Sy(f)}d_v\kakko(\LL^0f,\dev)^2 \\
&=
\left(\sum_{v\in\Px(f)\cup\Rx(f)}d_v\right)\kakko(\LL^0f,\dex)^2+\left(\sum_{v\in\Qy(f)\cup\Sy(f)}d_v\right)\kakko(\LL^0f,\dey)^2 &\Big(\text{by \autoref{key property}}\Big).
\end{align*}
Notice that by \autoref{grad lemma} and \autoref{key cor}, we have
\begin{equation*}
\sum_{v\in\Px(f)\cup\Rx(f)}\LL^0f(v)
=
w_{e_V}\max_{\mathsf{b}\in\mathsf{B}_{e_V}}\kakko(f,\mathsf{b})+w_e\max_{\mathsf{b}\in\mathsf{B}_e}\kakko(f,\mathsf{b})
=
w_{e_V}+w_e\max_{\mathsf{b}\in\mathsf{B}_e}\kakko(f,\mathsf{b}).
\end{equation*}
Hence, we have
\begin{equation*}
\dis \kakko(\LL^0f,\dex)
=
\frac{\dis\;w_{e_V}+w_e\max_{\mathsf{b}\in\mathsf{B}_e}\kakko(f,\mathsf{b})\;}{\dis\;\#\Px(f)\cdot d_x+\#\Rx(f)\cdot d^x\;},
\;\text{ where }\;
d^x
\:=
\left\{
\begin{aligned}
&\;w_{e_V} & (x\in e) , \vspace{1mm} \\
&\;w_{e_V}+w_e & (x\not\in e) . 
\end{aligned}
\right. 
\end{equation*}
Moreover, notice that the following hold:
\begin{itemize}
\item
$\dis \max_{\mathsf{b}\in\mathsf{B}_e}\kakko(f,\mathsf{b})=0$ if and only if $\dis \min_{v\in e}\kakko(f,\dev)=\max_{v\in e}\kakko(f,\dev)\in\big\{\kakko(f,\dey),\;\kakko(f,\dex)\big\}$.
\item
$\dis \max_{\mathsf{b}\in\mathsf{B}_e}\kakko(f,\mathsf{b})=1$ if and only if $\dis \min_{v\in e}\kakko(f,\dev)=\kakko(f,\dey)<\kakko(f,\dex)=\max_{v\in e}\kakko(f,\dev)$.
\end{itemize}

We define $g\in\Lipxy$ by
\begin{equation*}
\dis \kakko(g,\dev)
=
\left\{
\begin{aligned}
&\;\kakko(f,\dex) & (v\in\VV(f)) , \vspace{1mm} \\
&\;\kakko(f,\dev) & (v\not\in\VV(f)) . 
\end{aligned}
\right. 
\end{equation*}
By the definition of $g$, we have $\Px(g)\cup\Rx(g)=\Px(f)\cup\Rx(f)\cup\VV(f)$, $\Qy(g)=\Qy(f)$, $\Sy(g)=\Sy(f)$ and $\VV(g)=\emptyset$.
Moreover, we have $\max_{v\in e}\kakko(g,\dev)=\max_{v\in e}\kakko(f,\dev)$ and $\min_{v\in e}\kakko(g,\dev)=\min_{v\in e}\kakko(f,\dev)$.
Note that by the definition of $\LL^0$, $\kakko(\LL^0g,\dev)=\kakko(\LL^0f,\dev)$ holds for all $v\in\Qy(g)\cup\Sy(g)$.
In addition, by \autoref{key property}(1)(2), $\kakko(\LL^0g,\dev)$ is constant on $\Px(g)$ and $\Rx(g)$.
We denote these values by $\kakko(\LLP,\dex)$ and $\kakko(\LLR,\dex)$, respectively.
Then, since we observe
\begin{align}
\bullet\;\;
\sum_{v\in\Px(g)\cup\Rx(g)}\LL^0g(v)
&=
w_{e_V}+w_e\max_{\mathsf{b}\in\mathsf{B}_e}\kakko(g,\mathsf{b})
=
w_{e_V}+w_e\max_{\mathsf{b}\in\mathsf{B}_e}\kakko(f,\mathsf{b}), \label{constraint 1}
\\
\bullet\;\;
\sum_{v\in\Px(g)\cup\Rx(g)}\LL^0g(v)
&=
\sum_{v\in\Px(g)\cup\Rx(g)}d_v\kakko(\LL^0g,\dev)
=
\left(\sum_{v\in\Px(g)}d_v\right)\kakko(\LLP,\dex)+\left(\sum_{v\in\Rx(g)}d_v\right)\kakko(\LLR,\dex) \notag \\
&=
\#\Px(g)\cdot d_x\cdot\kakko(\LLP,\dex)+\#\Rx(g)\cdot d^x\cdot\kakko(\LLR,\dex),
\label{constraint 2}
\end{align}
we find
\begin{align*}
\|\LL^0g\|_{\Di}^2 
&=
\sum_{v\in\Px(g)\cup\Rx(g)}d_v\kakko(\LL^0g,\dev)^2+\sum_{v\in\Qy(g)\cup\Sy(g)}d_v\kakko(\LL^0g,\dev)^2 \\
&=
\left(\sum_{v\in\Px(g)}d_v\right)\kakko(\LLP,\dex)^2+\left(\sum_{v\in\Rx(g)}d_v\right)\kakko(\LLR,\dex)^2+\left(\sum_{v\in\Qy(f)\cup\Sy(f)}d_v\right)\kakko(\LL^0f,\dey)^2 \\
&=
\#\Px(g)\cdot d_x\cdot\kakko(\LLP,\dex)^2+\#\Rx(g)\cdot d^x\cdot\left\{\frac{\dis\;\Big(w_{e_V}+w_e\max_{\mathsf{b}\in\mathsf{B}_e}\kakko(f,\mathsf{b})\Big)-\#\Px(g)\cdot d_x\cdot\kakko(\LLP,\dex)\;}{\dis\#\Rx(g)\cdot d^x}\right\}^2 \\
&\;\quad 
\quad+\left(\sum_{v\in\Qy(f)\cup\Sy(f)}d_v\right)\kakko(\LL^0f,\dey)^2.
\end{align*}
Hence, we have
\begin{align}
\frac{\;\dd\|\LL^0g\|_{\Di}^2 \;}{\;\dd\kakko(\LLP,\dex)\;}
&=
2\#\Px(g)\cdot d_x\cdot\kakko(\LLP,\dex)-\frac{\dis\;2\#\Px(g)\cdot d_x\;}{\dis\;\#\Rx(g)\cdot d^x\;}\cdot\Bigg\{\bigg(w_{e_V}+w_e\max_{\mathsf{b}\in\mathsf{B}_e}\kakko(f,\mathsf{b})\bigg)-\#\Px(g)\cdot d_x\kakko(\LLP,\dex)\Bigg\} \notag \\
&=
2\#\Px(g)\cdot d_x\cdot\left\{\left(1+\frac{\dis\;\#\Px(g)\cdot d_x\;}{\dis\;\#\Rx(g)\cdot d^x\;}\right)\kakko(\LLP,\dex)-\frac{\dis\;w_{e_V}+w_e\max_{\mathsf{b}\in\mathsf{B}_e}\kakko(f,\mathsf{b})\;}{\dis\;\#\Rx(g)\cdot d^x\;}\right\} \notag \\
&=
\frac{\dis\;2\#\Px(g)\cdot d_x\cdot\Big(\#\Px(g)\cdot d_x+\#\Rx(g)\cdot d^x\Big)\;}{\dis\;\#\Rx(g)\cdot d^x\;}\cdot\left\{\kakko(\LLP,\dex)-\frac{\dis\;w_{e_V}+w_e\max_{\mathsf{b}\in\mathsf{B}_e}\kakko(f,\mathsf{b})\;}{\dis\;\#\Px(g)\cdot d_x+\#\Rx(g)\cdot d^x\;}\right\}. \label{LLP}
\end{align}

\textbf{\underline{Claim.}} 
\begin{equation*}
\mathsf{T}
\:=
\frac{\dis\;w_{e_V}+w_e\max_{\mathsf{b}\in\mathsf{B}_e}\kakko(g,\mathsf{b})\;}{\dis\;\#\Px(g)\cdot d_x+\#\Rx(g)\cdot d^x\;}
\in[0,1].
\end{equation*}

\textbf{\underline{Proof of Claim.}} 
By $\#\Px(g)\ge1$ since $x\in\Px(g)$, it is clear that if $d_x=w_{e_V}+w_e$, then $\mathsf{T}\in[0,1]$.
Assume that $d^x=w_{e_V}+w_e$.
If $\max_{\mathsf{b}\in\mathsf{B}_e}\kakko(g,\mathsf{b})=0$, then by $\#\Px(g)\ge1$, $\mathsf{T}\in[0,1]$.
If $\max_{\mathsf{b}\in\mathsf{B}_e}\kakko(g,\mathsf{b})=1$, then by the definition of $g$, $\Rx(g)\ge1$, thus, $\mathsf{T}\in[0,1]$.
$\blacksquare$
\vspace{2mm}

By \eqref{LLP}, \textbf{Claim} and the definition of $\LL^0g$, we find $\kakko(\LLP,\dex)=\mathsf{T}$.
Moreover, by \eqref{constraint 1} and \eqref{constraint 2}, we also find $\kakko(\LLR,\dex)=\mathsf{T}$.
Hence, we have $\kakko(\LL^0g,\dex)=\mathsf{T}$ and $\kakko(\LL^0g,\dey)=\kakko(\LL^0f,\dey)$.
Here, by $\VV(f)\neq\emptyset$, $\#\Px(g)>\#\Px(f)$ or $\#\Rx(g)>\#\Rx(g)$ holds.
Therefore, we obtain
\begin{align*}
\kakko(\LL^0g,\dex-\dey)
=
\mathsf{T}-\kakko(\LL^0f,\dey)
&=
\frac{\dis\;w_{e_V}+w_e\max_{\mathsf{b}\in\mathsf{B}_e}\kakko(f,\mathsf{b})\;}{\dis\;\#\Px(g)\cdot d_x+\#\Rx(g)\cdot d^x\;}-\kakko(\LL^0f,\dey) \\
&<
\frac{\dis\;w_{e_V}+w_e\max_{\mathsf{b}\in\mathsf{B}_e}\kakko(f,\mathsf{b})\;}{\dis\;\#\Px(f)\cdot d_x+\#\Rx(f)\cdot d^x\;}-\kakko(\LL^0f,\dey)
=
\kakko(\LL^0f,\dex-\dey).
\end{align*}
\end{proof}

\begin{cor} \label{flow of C}
Let $E=\{e_V,e\}$, where $e_V$ is a hyperedge including all vertices, and $x,y\in V$.
Then, we have
\begin{equation*}
\CC(x,y)
=
\min_{f\in\LIPxy}\kakko(\LL^0f,\dex-\dey), 
\end{equation*}
where
\begin{equation*}
\LIPxy
\:=
\Big\{f\in\Lipxy\;\Big|\;\kakko(f,\delta_v)=\kakko(f,\dex)\mathrm{\;or\;}\kakko(f,\delta_v)=\kakko(f,\dey)\mathrm{\;holds\;for\;all\;}v\in V\Big\}.
\end{equation*}
\end{cor}

For simplicity in the following discussion, we define as follows.

\begin{defi}[Flow]
For given $f\in\rv$, we call the steepest gradient $\flow(e)\in\mathrm{arg\,max}_{\mathsf{b}\in\mathsf{B_e}}\kakko(f,\mathsf{b})$ on each hyperedge $e$ the \emph{flow} on $e$ induced by $f$.
\end{defi}

\begin{nota}
In the following, we represent the flow of the steepest gradient $\flow(e)$ by the symbol of an arrow from $\posiflow(e)$ to $\negaflow(e)$.
For instance, in case we set $\flow(e)=(\dex+\dey)/2-\delta_v$ in $e=\{x,y,z,v\}$, then this flow is represented by
\begin{equation*}
\{x,y\} \quad \to \quad \{v\}.
\end{equation*}
\end{nota}

%%%%%%%%%%%%%%%%%%%%%%%%%%%%%%%%%%%%%%%%%%%%%%%%%%%%%%%%%%%%%%%%%%%%%%%%%%%
\subsection{$1$-regular hypergraphs} \label{1-regular subsection}
%%%%%%%%%%%%%%%%%%%%%%%%%%%%%%%%%%%%%%%%%%%%%%%%%%%%%%%%%%%%%%%%%%%%%%%%%%%

In this subsection, we calculate $\CC(R_{n,1})$.

\begin{exa} \label{1-regular ex}
\begin{equation}
\CC(R_{n,1})
= 
\frac{n}{\dis \;\upgauss(\frac{n}{2})\ungauss(\frac{n}{2})\;}. \label{1-regular kxy}
\end{equation}
\end{exa}

\begin{proof}
Let $V=\{x,y,v_1,\ldots\,,v_{n-2}\}$ be the vertex set of $R_{n,1}$.
Notice that since \autoref{key property} and \autoref{key cor} also hold for hypergraphs with $e=e_V$, in particular, $1$-regular hypergraphs, \autoref{flow of C} also holds for $1$-regular hypergraphs.

Suppose that a flow of $f\in\LIPxy$ on $e_V$ is represented by 
\begin{equation*}
\{x,\;\underbrace{\,\ldots\,,\;v_i,\;\ldots\,}_{I\text{ times}}\}
\quad \to \quad 
\{y,\;\underbrace{\,\ldots\,,\;v_j,\;\ldots\,}_{J\text{ times}}\}, 
\end{equation*}
where $I,J\in\N_0$ satisfying $I+J=n-2$.
Then, by \eqref{L0 formula} and \autoref{key property}, we have
\begin{equation*}
\left\{
\begin{aligned}
&\;\dis \LL^0f(x)+I\LL^0f(v_i)\;=\;w_{e_V} , \vspace{1mm} \\
&\;\dis \LL^0f(y)+J\LL^0f(v_j)=-w_{e_V},
\end{aligned}
\right. 
\qquad\text{ and }\qquad
\left\{
\begin{aligned}
&\;\dis \kakko(\LL^0f,\dex)=\frac{1}{\;I+1\;} , \vspace{1mm} \\
&\;\dis \kakko(\LL^0f,\dey)=-\frac{1}{\;J+1\;}.
\end{aligned}
\right. 
\end{equation*}
Thus, we find
\begin{equation*}
\CC(x,y)
=
\min_{f\in\LIPxy}\kakko(\LL^0f,\dex-\dey)
=
\min_{\substack{I,J\in\N_0 \\ I+J=n-2}}\left\{\frac{1}{I+1}+\frac{1}{J+1}\right\}
=
\frac{n}{\;\dis\upgauss(\frac{n}{2})\ungauss(\frac{n}{2})\;}.
\end{equation*}
\end{proof}

\begin{rem} \label{relative heat}
By \autoref{1-regular ex}, $\K(R_{2,1})=2$ holds.
Moreover, since $R_{2,1}$ is a (weighted) complete graph $K_2$, $\K_{\mathrm{LLY}}(K_2)=\K_\mathrm{IKTU}(R_{2,1})$ holds by \cite[Proposition 4.1]{IKTU}.
Therefore, we obtain $\K(R_{2,1})=\K_{\mathrm{LLY}}(K_2)=\K_\mathrm{IKTU}(R_{2,1})=2$ by \autoref{example of graph}. 
In addition, note that 
the (hyper)edge weight do not affect these curvatures in the case $\#E=1$.
\end{rem}

\begin{proof}[\underline{Proof of \autoref{1-regular limit}}]
\hspace{0mm}
\begin{itemize}
\item[$(1)$] 
By the previous discussion, we obtain
\begin{equation*}
\kIKTU(R_{n,1})
\stackrel{\text{\eqref{kIKTU < k}}}{\le}
\K(R_{n,1})
\stackrel{\text{\eqref{C>Kxy}}}{\le}
\CC(R_{n,1})
\stackrel{\text{\eqref{1-regular kxy}}}{=}
\frac{n}{\dis \;\upgauss(\frac{n}{2})\ungauss(\frac{n}{2})\;}
\xrightarrow{\;n\to\infty\;}
0.
\end{equation*}
\item[$(2)$] 
This follows from \autoref{main theorem} and $(1)$. 
\end{itemize}
\end{proof}

%%%%%%%%%%%%%%%%%%%%%%%%%%%%%%%%%%%%%%%%%%%%%%%%%%%%%%%%%%%%%%%%%%%%%%%%%%%
\subsection{Hypergraphs satisfying $e_V\in E$ with $2$ hyperedges} \label{|E|=2 subsection}
%%%%%%%%%%%%%%%%%%%%%%%%%%%%%%%%%%%%%%%%%%%%%%%%%%%%%%%%%%%%%%%%%%%%%%%%%%%

In this subsection, we list the calculation results of $\CC(x,y)$ of hypergraphs obtained by adding one hyperedge to a $1$-regular hypergraph.
We denote the hyperedge that includes all vertices by $e_V$, and the other hyperedge by $e$.
According to the relationship between the added hyperedge $e$ and $x,y\in V$, we calculate $\CC(x,y)$ in the three types of hypergraphs \autoref{V=n, E=XYZ}, \autoref{V=n, E=XZZ} and \autoref{V=n, E=ZZ}.
Here, we set $A,B\in\N_0$ such that $A+B\ge1$, $V=\{x,y,p_1,\ldots,p_A,q_1,\ldots,q_B\}$ with $n=A+B+2\,(\ge3)$ and $A\ge1$ in \autoref{V=n, E=ZZ}.

\vspace{-10mm}
\begin{figure}[H]
\begin{tabular}{ccccc} 
\begin{minipage}{0.3\hsize}
\centering
\begin{tikzpicture}
\node (X) at (0.25,0.7) {$x$};
\node (Y) at (1.75,0.7) {$y$};
\node (z) at (0.25,0) {$p_1$};
\node (z) at (1,0) {$\cdots$};
\node (z) at (1.75,0) {$p_A$};
\node (w) at (0,-1) {$q_1$};
\node (w) at (1,-1) {$\cdots$};
\node (w) at (2,-1) {$q_B$};
\draw[ultra thick] (-0.5,-1.3)--(-0.5,1.2);
\draw[ultra thick] (2.5,-1.3)--(2.5,1.2);
\draw[ultra thick] (-0.5,1.2)--(2.5,1.2);
\draw[ultra thick] (-0.5,-1.3)--(2.5,-1.3);
\draw[red, ultra thick] (-0.1,-0.3)--(-0.1,1);
\draw[red, ultra thick] (2.1,-0.3)--(2.1,1);
\draw[red, ultra thick] (-0.1,-0.3)--(2.1,-0.3);
\draw[red, ultra thick] (-0.1,1)--(2.1,1);
\node (empty) at (1,2.5) {};
\end{tikzpicture}
\caption{$e=\{x,y,p_1,\ldots,p_A\}$.} \label{V=n, E=XYZ}
\end{minipage}
\begin{minipage}{0.05\hsize}
\end{minipage}
\begin{minipage}{0.3\hsize}
\centering
\begin{tikzpicture}
\node (X) at (0.25,0.7) {$x$};
\node (Y) at (1.75,0.7) {$y$};
\node (z) at (0.25,0) {$p_1$};
\node (z) at (1,0) {$\cdots$};
\node (z) at (1.75,0) {$p_A$};
\node (w) at (0,-1) {$q_1$};
\node (w) at (1,-1) {$\cdots$};
\node (w) at (2,-1) {$q_B$};
\draw[ultra thick] (-0.5,-1.3)--(-0.5,1.2);
\draw[ultra thick] (2.5,-1.3)--(2.5,1.2);
\draw[ultra thick] (-0.5,1.2)--(2.5,1.2);
\draw[ultra thick] (-0.5,-1.3)--(2.5,-1.3);
\draw[red, ultra thick] (-0.1,-0.3)--(-0.1,1);
\draw[red, ultra thick] (2.1,-0.3)--(2.1,0.2);
\draw[red, ultra thick] (0.5,1)--(2.1,0.2);
\draw[red, ultra thick] (-0.1,-0.3)--(2.1,-0.3);
\draw[red, ultra thick] (-0.1,1)--(0.5,1);
\node (empty) at (1,2.5) {};
\end{tikzpicture}
\caption{$e=\{x,p_1,\ldots,p_A\}$.} \label{V=n, E=XZZ}
\end{minipage}
\begin{minipage}{0.05\hsize}
\end{minipage}
\begin{minipage}{0.3\hsize}
\centering
\begin{tikzpicture}
\node (X) at (0.25,0.7) {$x$};
\node (Y) at (1.75,0.7) {$y$};
\node (z) at (0.25,0) {$p_1$};
\node (z) at (1,0) {$\cdots$};
\node (z) at (1.75,0) {$p_A$};
\node (w) at (0,-1) {$q_1$};
\node (w) at (1,-1) {$\cdots$};
\node (w) at (2,-1) {$q_B$};
\draw[ultra thick] (-0.5,-1.3)--(-0.5,1.2);
\draw[ultra thick] (2.5,-1.3)--(2.5,1.2);
\draw[ultra thick] (-0.5,1.2)--(2.5,1.2);
\draw[ultra thick] (-0.5,-1.3)--(2.5,-1.3);
\draw[red, ultra thick] (-0.1,-0.3)--(-0.1,0.3);
\draw[red, ultra thick] (2.1,-0.3)--(2.1,0.3);
\draw[red, ultra thick] (-0.1,-0.3)--(2.1,-0.3);
\draw[red, ultra thick] (-0.1,0.3)--(2.1,0.3);
\node (empty) at (1,2.5) {};
\end{tikzpicture}
\caption{$e=\{p_1,\ldots,p_A\}$.} \label{V=n, E=ZZ}
\end{minipage}
\end{tabular}
\end{figure}

\begin{exa} \label{main example}
$\CC(x,y)$ in each of the hypergraphs in \autoref{V=n, E=XYZ}, \autoref{V=n, E=XZZ} and \autoref{V=n, E=ZZ} is as follows.
\begin{itemize}
\item \autoref{V=n, E=XYZ}:
\begin{equation}
\CC(x,y)
=
(w_{e_V}+w_e)\cdot
\frac{\vol(H)}{\;\dis\max_{\substack{I,J\in\N,\,K,L\in\N_0 \\ I+J=A+2 \\ K+L=B}}\Big\{(w_{e_V}+w_e)I+w_{e_V}K\Big\}\Big\{(w_{e_V}+w_e)J+w_{e_V}L\Big\}\;}. \label{xyp kxy weighted}
\end{equation}
\item \autoref{V=n, E=XZZ}:
\begin{equation}
\CC(x,y)
=
\min\left\{
\begin{aligned}
&\;\dis \frac{\vol(H)}{\;\dis\max_{\substack{K,L\in\N_0 \\ K+L=B}}\Big\{(w_{e_V}+w_e)(A+1)+w_{e_V}K\Big\}(L+1)\;} {\Large \textbf{,}}\; \vspace{2mm} \\ 
&\;\dis (w_{e_V}+w_e)\cdot
\frac{\vol(H)}{\;\dis\max_{\substack{I,J\in\N,\,K,L\in\N_0 \\ I+J=A+1 \\ K+L=B}}\dis\Big\{(w_{e_V}+w_e)I+w_{e_V}K\Big\}\Big\{(w_{e_V}+w_e)J+w_{e_V}(L+1)\Big\}\;}
\end{aligned}
\right\}. \label{xpp kxy weighted}
\end{equation}
\item \autoref{V=n, E=ZZ}:
\begin{equation}
\CC(x,y)
=
\min\left\{
\begin{aligned}
&\;\dis \frac{\vol(H)}{\;\dis\max_{\substack{K,L\in\N_0 \\ K+L=B}}\Big\{(w_{e_V}+w_e)A+w_{e_V}(K+1)\Big\}(L+1)\;} {\Large \textbf{,}}\; \vspace{2mm} \\ 
&\;\dis (w_{e_V}+w_e)\cdot
\frac{\vol(H)}{\;\dis\max_{\substack{I,J\in\N,\,K,L\in\N_0 \\ I+J=A \\ K+L=B}}\Big\{(w_{e_V}+w_e)I+w_{e_V}(K+1)\Big\}\Big\{(w_{e_V}+w_e)J+w_{e_V}(L+1)\Big\}\;}\; 
\end{aligned}
\right\}. \label{pp kxy weighted}
\end{equation}
\end{itemize}
\end{exa}

The detailed computation of \autoref{main example} is deferred to \autoref{calculation section}.

\begin{rem}
In each of the following situations, we can see that $\CC(x,y)$ in \autoref{main example} is reduced to that in \autoref{1-regular ex}.
We denote a hypergraph with $n$ vertices having two different hyperedges that contain all vertices by $R_{n,2}$.
\begin{itemize}
\item 
In \autoref{V=n, E=XYZ}, set $B=0$. 
Then, this hypergraph is $R_{{A+2},2}$ having multi-hyperedges with weights $w_{e_V}$ and $w_e$.
Substituting $B=0$ into \eqref{xyp kxy weighted} yields 
\begin{equation*}
\CC(x,y)
=
(w_{e_V}+w_e)\cdot
\frac{(w_{e_V}+w_e)(A+2)}{\;\dis\max_{\substack{I,J\in\N \\ I+J=A+2}}(w_{e_V}+w_e)I\cdot(w_{e_V}+w_e)J\;}
=
\frac{A+2}{\;\dis\upgauss(\frac{A+2}{2})\ungauss(\frac{A+2}{2})\;}.
\end{equation*}
This coincides with $\CC(x,y)$ of $R_{{A+2},1}$ having a single hyperedge with weight $w_{e_V}+w_e$.
\item 
In \autoref{V=n, E=XYZ}, when the weight $w_{e_V}$ approaches $0$, in the limit we have $R_{{A+2},1}$ having a hyperedge $e$ with weight $w_e$.
Hence, we expect $\lim_{w_{e_V}\dto0}\CC(x,y)=\CC(R_{{A+2},1})$, and it is indeed the case by \eqref{xyp kxy weighted}:
\begin{equation*}
\lim_{w_{e_V}\dto0}\CC(x,y)
=
w_e\cdot
\frac{\vol(H)}{\;\dis\max_{\substack{I,J\in\N \\ I+J=A+2}}w_eI\cdot w_eJ\;}
=
\frac{A+2}{\;\dis\upgauss(\frac{A+2}{2})\ungauss(\frac{A+2}{2})\;}.
\end{equation*}
\item 
In \autoref{V=n, E=ZZ}, set $A=0$. 
Then this hypergraph is $R_{{B+2},1}$ having a single hyperedge with weight $w_{e_V}$.
Substituting $A=0$ and $w_e=0$ into \eqref{pp kxy weighted} yields 
\begin{equation*}
\CC(x,y)
=
\min\left\{
\begin{aligned}
\;\dis \frac{w_{e_V}(B+2)}{\;\dis\max_{\substack{K,L\in\N_0 \\ K+L=B}}w_{e_V}(K+1)(L+1)\;} 
{\Large \textbf{,}}\;\; 
w_{e_V}\cdot
\frac{w_{e_V}(B+2)}{\;\dis\max_{\substack{K,L\in\N_0 \\ K+L=B}}w_{e_V}(K+1)\cdot w_{e_V}(L+1)\;}\; 
\end{aligned}
\right\} 
=
\frac{B+2}{\;\dis\upgauss(\frac{B+2}{2})\ungauss(\frac{B+2}{2})\;}.
\end{equation*}
\end{itemize}
\end{rem}

\begin{cor}
For hypergraphs $H_1=\big(\{x,y,z\},\,\{xyz\}\big)$ (\autoref{E=1}) and $H_2=\big(\{x,y,z\},\,\{xyz,xy\}\big)$ (\autoref{E=2}), we have
\begin{equation*}
\CC(x,y)
=
\dis\frac{\vol(H)}{\;\max\{d_x,d_y\}+d_z\;}
,\quad
\CC(y,z)
=
\dis\frac{\vol(H)}{\;\max\{d_y,d_z\}+d_x\;}.
\end{equation*}
In particular, if the hypergraphs $H_1$ and $H_2$ are unweighted, then we have
\begin{equation*}
\mathrm{\autoref{E=1}:}\;\;
\CC(x,y)=\CC(y,z)=\frac{3}{2}
,\qquad
\mathrm{\autoref{E=2}:}\;\;
\CC(x,y)=\frac{5}{3}
\quad\mathrm{ and }\quad
\CC(y,z)=\frac{5}{4}.
\end{equation*}
\end{cor}

\vspace{-8mm}
\begin{figure}[H]
\begin{tabular}{ccc}
\begin{minipage}{0.47\hsize}
\begin{center}
\begin{tikzpicture}
\node (v) at (1,2) {};
\node (X) at (1,1) {$x$};
\node (Y) at (0,0) {$y$};
\node (Z) at (2,0) {$z$};
\draw[ultra thick] (-0.5,-0.5)--(-0.5,1.5);
\draw[ultra thick] (2.5,-0.5)--(2.5,1.5);
\draw[ultra thick] (-0.5,1.5)--(2.5,1.5);
\draw[ultra thick] (-0.5,-0.5)--(2.5,-0.5);
\end{tikzpicture}
\caption{A hypergraph $H_1$.} \label{E=1}
\end{center}
\end{minipage}
%%%%%%%%%%%%%%%%%%%%%%%%%%%%%%%%%%%%%%%%%%%%%%%%%%%%%%%%%%%%%%%%%%%%%%%%%%%%%%%%%%%%%%%%%%%%%%%%%%%%
\begin{minipage}{0.05\hsize}
\begin{center}
\end{center}
\end{minipage}
%%%%%%%%%%%%%%%%%%%%%%%%%%%%%%%%%%%%%%%%%%%%%%%%%%%%%%%%%%%%%%%%%%%%%%%%%%%%%%%%%%%%%%%%%%%%%%%%%%%%
\begin{minipage}{0.47\hsize}
\begin{center}
\begin{tikzpicture}
\node (v) at (1,2) {};
\node (X) at (1,1) {$x$};
\node (Y) at (0,0) {$y$};
\node (Z) at (2,0) {$z$};
\draw[ultra thick] (-0.5,-0.5)--(-0.5,1.5);
\draw[ultra thick] (2.5,-0.5)--(2.5,1.5);
\draw[ultra thick] (-0.5,1.5)--(2.5,1.5);
\draw[ultra thick] (-0.5,-0.5)--(2.5,-0.5);
\draw[red, ultra thick] (0.2,0.2)--(0.8,0.8); 
\end{tikzpicture}
\caption{A hypergraph $H_2$.} \label{E=2}
\end{center}
\end{minipage}
\end{tabular}
\end{figure}

We can represent the formulae \eqref{xyp kxy weighted}, \eqref{xpp kxy weighted} and \eqref{pp kxy weighted} by the follwing single equation:
\begin{equation*}
\CC(x,y)
=
\min_{f\in\LIPxy}
\left\{\sum_{e\in E}
w_e\left(\max_{\mathsf{b}\in\mathsf{B}_e}\kakko(f,\mathsf{b})\right)\cdot
\frac{\vol(H)}{\;\dis\max_{\flow(e)\in\mathsf{B}_e}\left(\sum_{v^+\in\posiflow(e)}d_{v^+}\right)\left(\sum_{v^-\in\negaflow(e)}d_{v^-}\right)\;}\right\},
\end{equation*}
where by \autoref{key cor}, we have
\begin{equation*}
\max_{\mathsf{b}\in\mathsf{B}_{e_V}}\kakko(f,\mathsf{b})=1,
\quad
\max_{\mathsf{b}\in\mathsf{B}_e}\kakko(f,\mathsf{b})\in\{0,1\}.
\end{equation*}

%%%%%%%%%%%%%%%%%%%%%%%%%%%%%%%%%%%%%%%%%%%%%%%%%%%%%%%%%%%%%%%%%%%%%%%%%%%
%%%%%%%%%%%%%%%%%%%%%%%%%%%%%%%%%%%%%%%%%%%%%%%%%%%%%%%%%%%%%%%%%%%%%%%%%%%
\section{Appendix: Concrete computations of $\CC(x,y)$} \label{calculation section}
%%%%%%%%%%%%%%%%%%%%%%%%%%%%%%%%%%%%%%%%%%%%%%%%%%%%%%%%%%%%%%%%%%%%%%%%%%%
%%%%%%%%%%%%%%%%%%%%%%%%%%%%%%%%%%%%%%%%%%%%%%%%%%%%%%%%%%%%%%%%%%%%%%%%%%%

In this section, we calculate $\CC(x,y)$ in each hypergraph of \autoref{V=n, E=XYZ}, \autoref{V=n, E=XZZ} and \autoref{V=n, E=ZZ}.
Hereafter, we suppose that $I,J,K,L\in\N_0$, $I+J=A$ and $K+L=B$ unless otherwise noted.

%%%%%%%%%%%%%%%%%%%%%%%%%%%%%%%%%%%%%%%%%%%%%%%%%%%%%%%%%%%%%%%%%%%%%%%%%%%
\subsection{A hypergraph in \autoref{V=n, E=XYZ}} \label{appendix: V=n, E=XYZ}
%%%%%%%%%%%%%%%%%%%%%%%%%%%%%%%%%%%%%%%%%%%%%%%%%%%%%%%%%%%%%%%%%%%%%%%%%%%

Consider a flow on $e_V$ induced from some $f\in\LIPxy$ represented by
\begin{equation}
\{x,\;\underbrace{\,\ldots\,,\;p_i,\;\ldots\,}_{I\text{ times}} ,\;\underbrace{\,\ldots\,,\;q_k,\;\ldots\,}_{K\text{ times}}\}
\quad \to \quad 
\{y,\;\underbrace{\,\ldots\,,\;p_j,\;\ldots\,}_{J\text{ times}} ,\;\underbrace{\,\ldots\,,\;q_\ell,\;\ldots\,}_{L\text{ times}}\}.
\label{flow}
\end{equation}
Then, the associated flow on $e$ is represented by
\begin{align*}
\{x,\;\underbrace{\,\ldots\,,\;p_i,\;\ldots\,}_{I\text{ times}}\}
\quad \to \quad 
\{y,\;\underbrace{\,\ldots\,,\;p_j,\;\ldots\,}_{J\text{ times}}\}.
\end{align*}
By \eqref{L0 formula} and \autoref{key property}, we have
\begin{align*}
\left\{
\begin{aligned}
&\;\dis \LL^0f(x)+I\LL^0f(p_i)+K\LL^0f(q_k)\;=\;w_{e_V}+w_e , \vspace{1mm} \\
&\;\dis \LL^0f(y)+J\LL^0f(p_j)+L\LL^0f(q_\ell)=-(w_{e_V}+w_e),
\end{aligned}
\right. 
\qquad\text{ and }\qquad
\left\{
\begin{aligned}
&\;\dis \kakko(\LL^0f,\dex)=\frac{w_{e_V}+w_e}{\;(w_{e_V}+w_e)(I+1)+w_{e_V}K\;}, \vspace{1mm} \\
&\;\dis \kakko(\LL^0f,\dey)=-\frac{w_{e_V}+w_e}{\;(w_{e_V}+w_e)(J+1)+w_{e_V}L\;}.
\end{aligned}
\right. 
\end{align*}
Thus, by \autoref{flow of C}, we obtain \eqref{xyp kxy weighted}.
Indeed, 
\begin{align*}
\CC(x,y)
&=
(w_{e_V}+w_e)\cdot\min_{\substack{I,J,K,L\in\N_0 \\ I+J=A \\ K+L=B}}\left\{\frac{1}{\;(w_{e_V}+w_e)(I+1)+w_{e_V}K\;}+\frac{1}{\;(w_{e_V}+w_e)(J+1)+w_{e_V}L\;}\right\} \\
&=
(w_{e_V}+w_e)\cdot\min_{\substack{I,J,K,L\in\N_0 \\ I+J=A \\ K+L=B}}
\frac{(w_{e_V}+w_e)(A+2)+w_{e_V}B}{\;\dis\Big\{(w_{e_V}+w_e)(I+1)+w_{e_V}K\Big\}\Big\{(w_{e_V}+w_e)(J+1)+w_{e_V}L\Big\}\;} \\
&=
(w_{e_V}+w_e)\cdot
\frac{\vol(H)}{\;\dis\max_{\substack{I,J\in\N,\,K,L\in\N_0 \\ I+J=A+2 \\ K+L=B}}\Big\{(w_{e_V}+w_e)I+w_{e_V}K\Big\}\Big\{(w_{e_V}+w_e)J+w_{e_V}L\Big\}\;}.
\end{align*}

%%%%%%%%%%%%%%%%%%%%%%%%%%%%%%%%%%%%%%%%%%%%%%%%%%%%%%%%%%%%%%%%%%%%%%%%%%%
\subsection{A hypergraph in \autoref{V=n, E=XZZ}} \label{subsub calculation E=XZZ}
%%%%%%%%%%%%%%%%%%%%%%%%%%%%%%%%%%%%%%%%%%%%%%%%%%%%%%%%%%%%%%%%%%%%%%%%%%%

Consider a flow of $f\in\LIPxy$ on $e_V$ represented by \eqref{flow}.
By \autoref{key cor}, it is sufficient to consider the following cases:
\begin{itemize}
\item[\textbf{(A)}] $\max_{\mathsf{b}\in\mathsf{B}_e}\kakko(f,\mathsf{b})=0$.
\item[\textbf{(B)}] $\max_{\mathsf{b}\in\mathsf{B}_e}\kakko(f,\mathsf{b})=1$.
\end{itemize}

\vspace{1mm}
\textbf{(A)} \underline{The case $\max_{\mathsf{b}\in\mathsf{B}_e}\kakko(f,\mathsf{b})=0$}: 
Since there is no flow on $e$ associated with the above flow, we find $(I,J)=(A,0)$.
Hence, by \eqref{L0 formula} and \autoref{key property}, we have
\begin{align*}
\left\{
\begin{aligned}
&\;\dis \LL^0f(x)+A\LL^0f(p_i)+K\LL^0f(q_k)\;=\;w_{e_V} , \vspace{1mm} \\
&\;\dis \LL^0f(y)+L\LL^0f(q_\ell)=-w_{e_V},
\end{aligned}
\right. 
\qquad\text{ and }\qquad
\left\{
\begin{aligned}
&\;\dis \kakko(\LL^0f,\dex)=\frac{w_{e_V}}{\;(w_{e_V}+w_e)(A+1)+w_{e_V}K\;} , \vspace{1mm} \\
&\;\dis \kakko(\LL^0f,\dey)=-\frac{1}{\;L+1\;}.
\end{aligned}
\right. 
\end{align*}
Thus, we obtain
\begin{align}
\min_{\substack{K,L\in\N_0 \\ K+L=B}}
\left\{\frac{w_{e_V}}{\;(w_{e_V}+w_e)(A+1)+w_{e_V}K\;}+\frac{1}{\;L+1\;}\right\} 
&=
\min_{\substack{K,L\in\N_0 \\ K+L=B}}
\frac{\;(w_{e_V}+w_e)(A+1)+w_{e_V}(B+1)\;}{\;\dis\Big\{(w_{e_V}+w_e)(A+1)+w_{e_V}K\Big\}(L+1)\;} \notag \\
&=
\frac{\vol(H)}{\;\dis\max_{\substack{K,L\in\N_0 \\ K+L=B}}\Big\{(w_{e_V}+w_e)(A+1)+w_{e_V}K\Big\}(L+1)\;}. \label{Fig2 1}
\end{align}

\vspace{1mm}
\textbf{(B)} \underline{The case $\max_{\mathsf{b}\in\mathsf{B}_e}\kakko(f,\mathsf{b})=1$}: 
The flow on $e$ associated with the flow \eqref{flow} is represented by
\begin{equation*}
\{x,\;\underbrace{\,\ldots\,,\;p_i,\;\ldots\,}_{I\text{ times}}\}
\quad \to \quad 
\{\underbrace{\,\ldots\,,\;p_j,\;\ldots\,}_{J\text{ times}}\},
\end{equation*}
where $J\ge1$ and $A\ge1$.
Hence, by \eqref{L0 formula} and \autoref{key property}, we have
\begin{align*}
\left\{
\begin{aligned}
&\;\dis \LL^0f(x)+I\LL^0f(p_i)+K\LL^0f(q_k)\;=\;w_{e_V}+w_e , \vspace{1mm} \\
&\;\dis \LL^0f(y)+J\LL^0f(p_j)+L\LL^0f(q_\ell)=-(w_{e_V}+w_e),
\end{aligned}
\right. 
\qquad\text{ and }\qquad
\left\{
\begin{aligned}
&\;\dis \kakko(\LL^0f,\dex)=\frac{w_{e_V}+w_e}{\;(w_{e_V}+w_e)(I+1)+w_{e_V}K\;}, \vspace{1mm} \\
&\;\dis \kakko(\LL^0f,\dey)=-\frac{w_{e_V}+w_e}{\;(w_{e_V}+w_e)J+w_{e_V}(L+1)\;}.
\end{aligned}
\right. 
\end{align*}
Thus, we obtain
\begin{align}
&\;\quad 
(w_{e_V}+w_e)\cdot\min_{\substack{I,J,K,L\in\N_0 \\ J\ge1,\,I+J=A \\ K+L=B}}
\left\{\frac{1}{\;(w_{e_V}+w_e)(I+1)+w_{e_V}K\;}+\frac{1}{\;(w_{e_V}+w_e)J+w_{e_V}(L+1)\;}\right\} \notag \\
&=
(w_{e_V}+w_e)\cdot\min_{\substack{I,J,K,L\in\N_0 \\ J\ge1,\,I+J=A \\ K+L=B}}
\frac{(w_{e_V}+w_e)(A+1)+w_{e_V}(B+1)}{\;\dis\Big\{(w_{e_V}+w_e)(I+1)+w_{e_V}K\Big\}\Big\{(w_{e_V}+w_e)J+w_{e_V}(L+1)\Big\}\;} \notag \\
&=
(w_{e_V}+w_e)\cdot
\frac{\vol(H)}{\;\dis\max_{\substack{I,J\in\N,\,K,L\in\N_0 \\ I+J=A+1 \\ K+L=B}}\dis\Big\{(w_{e_V}+w_e)I+w_{e_V}K\Big\}\Big\{(w_{e_V}+w_e)J+w_{e_V}(L+1)\Big\}\;}. \label{Fig2 2}
\end{align}

Therefore, by \autoref{flow of C}, \eqref{Fig2 1} and \eqref{Fig2 2}, we obtain \eqref{xpp kxy weighted}.

%%%%%%%%%%%%%%%%%%%%%%%%%%%%%%%%%%%%%%%%%%%%%%%%%%%%%%%%%%%%%%%%%%%%%%%%%%%
\subsection{A hypergraph in \autoref{V=n, E=ZZ}} \label{subsub calculation E=ZZ}
%%%%%%%%%%%%%%%%%%%%%%%%%%%%%%%%%%%%%%%%%%%%%%%%%%%%%%%%%%%%%%%%%%%%%%%%%%%

We again consider a flow on $e_V$ represented by \eqref{flow}.
Similarly to \autoref{subsub calculation E=XZZ}, we consider the following cases:
\begin{itemize}
\item[\textbf{(A)}] $\max_{\mathsf{b}\in\mathsf{B}_e}\kakko(f,\mathsf{b})=0$.
\item[\textbf{(B)}] $\max_{\mathsf{b}\in\mathsf{B}_e}\kakko(f,\mathsf{b})=1$.
\end{itemize}

\vspace{1mm}
\textbf{(A)} \underline{The case $\max_{\mathsf{b}\in\mathsf{B}_e}\kakko(f,\mathsf{b})=0$}: 
Since there is no flow on $e$ associated with the above flow, we have $IJ=0$.
Note that by \autoref{key cor}, $\kakko(f,\dev)=\kakko(f,\dex)$ or $\kakko(f,\dev)=\kakko(f,\dey)$ holds for all $v\in e$.
\begin{itemize}
\item 
The case $\kakko(f,\dev)=\kakko(f,\dex)$: 
We find $(I,J)=(A,0)$.
Hence, by \eqref{L0 formula} and \autoref{key property}, we have
\begin{align*}
\left\{
\begin{aligned}
&\;\dis \LL^0f(x)+A\LL^0f(p_i)+K\LL^0f(q_k)\;=\;w_{e_V} , \vspace{1mm} \\
&\;\dis \LL^0f(y)+L\LL^0f(q_\ell)=-w_{e_V},
\end{aligned}
\right. 
\qquad\text{ and }\qquad
\left\{
\begin{aligned}
&\;\dis \kakko(\LL^0f,\dex)
=
\frac{w_{e_V}}{\;(w_{e_V}+w_e)A+w_{e_V}(K+1)\;} , \vspace{1mm} \\
&\;\dis \kakko(\LL^0f,\dey)
=
-\frac{1}{\;L+1\;}.
\end{aligned}
\right. 
\end{align*}
Thus, we obtain
\begin{align}
\min_{\substack{K,L\in\N_0 \\ K+L=B}}
\left\{\frac{w_{e_V}}{\;(w_{e_V}+w_e)A+w_{e_V}(K+1)\;}+\frac{1}{\;L+1\;}\right\} 
&=
\min_{\substack{K,L\in\N_0 \\ K+L=B}}
\frac{\;(w_{e_V}+w_e)A+w_{e_V}(B+2)\;}{\;\dis\Big\{(w_{e_V}+w_e)A+w_{e_V}(K+1)\Big\}(L+1)\;} \notag \\
&=
\frac{\vol(H)}{\;\dis\max_{\substack{K,L\in\N_0 \\ K+L=B}}\Big\{(w_{e_V}+w_e)A+w_{e_V}(K+1)\Big\}(L+1)\;}. \label{Fig3 1}
\end{align}
\item
The case $\kakko(f,\dev)=\kakko(f,\dey)$: 
We find $(I,J)=(0,A)$.
Hence, by \eqref{L0 formula} and \autoref{key property}, we have
\begin{align*}
\left\{
\begin{aligned}
&\;\dis \LL^0f(x)+K\LL^0f(q_k)\;=\;w_{e_V} , \vspace{1mm} \\
&\;\dis \LL^0f(y)+A\LL^0f(p_j)+L\LL^0f(q_\ell)=-w_{e_V},
\end{aligned}
\right. 
\qquad\text{ and }\qquad
\left\{
\begin{aligned}
&\;\dis \kakko(\LL^0f,\dex)
=
\frac{1}{\;K+1\;} , \vspace{1mm} \\
&\;\dis \kakko(\LL^0f,\dey)
=
-\frac{w_{e_V}}{\;(w_{e_V}+w_e)A+w_{e_V}(L+1)\;}.
\end{aligned}
\right. 
\end{align*}
Thus, we obtain
\begin{align*}
\min_{\substack{K,L\in\N_0 \\ K+L=B}}
\left\{\frac{1}{\;K+1\;}+\frac{w_{e_V}}{\;(w_{e_V}+w_e)A+w_{e_V}(L+1)\;}\right\} 
&=
\min_{\substack{K,L\in\N_0 \\ K+L=B}}
\frac{\;(w_{e_V}+w_e)A+w_{e_V}(B+2)\;}{\;\dis(K+1)\Big\{(w_{e_V}+w_e)A+w_{e_V}(L+1)\Big\}\;} \\
&=
\frac{\vol(H)}{\;\dis\max_{\substack{K,L\in\N_0 \\ K+L=B}}(K+1)\Big\{(w_{e_V}+w_e)A+w_{e_V}(L+1)\Big\}\;} \\
&=
\eqref{Fig3 1}.
\end{align*}
\end{itemize}

\vspace{1mm}
\textbf{(B)} \underline{The case $\max_{\mathsf{b}\in\mathsf{B}_e}\kakko(f,\mathsf{b})=1$}: 
The flow on $e$ associated with the flow \eqref{flow} is represented by
\begin{equation*}
\{\;\underbrace{\,\ldots\,,\;p_i,\;\ldots\,}_{I\text{ times}}\}
\quad \to \quad 
\{\underbrace{\,\ldots\,,\;p_j,\;\ldots\,}_{J\text{ times}}\},
\end{equation*}
where $I,J\ge1$ and $A\ge2$.
Hence, by \eqref{L0 formula} and \autoref{key property}, we have
\begin{align*}
\left\{
\begin{aligned}
&\;\dis \LL^0f(x)+I\LL^0f(p_i)+K\LL^0f(q_k)\;=\;w_{e_V}+w_e , \vspace{1mm} \\
&\;\dis \LL^0f(y)+J\LL^0f(p_j)+L\LL^0f(q_\ell)=-(w_{e_V}+w_e),
\end{aligned}
\right. 
\qquad\text{ and }\qquad
\left\{
\begin{aligned}
&\;\dis \kakko(\LL^0f,\dex)=\frac{w_{e_V}+w_e}{\;(w_{e_V}+w_e)I+w_{e_V}(K+1)\;}, \vspace{1mm} \\
&\;\dis \kakko(\LL^0f,\dey)=-\frac{w_{e_V}+w_e}{\;(w_{e_V}+w_e)J+w_{e_V}(L+1)\;}.
\end{aligned}
\right. 
\end{align*}
Thus, we obtain
\begin{align}
&\;\quad 
(w_{e_V}+w_e)\cdot\min_{\substack{I,J,K,L\in\N_0 \\ I,J\ge1,\,I+J=A \\ K+L=B}}
\left\{\frac{1}{\;(w_{e_V}+w_e)I+w_{e_V}(K+1)\;}+\frac{1}{\;(w_{e_V}+w_e)J+w_{e_V}(L+1)\;}\right\} \notag \\
&=
(w_{e_V}+w_e)\cdot\min_{\substack{I,J,K,L\in\N_0 \\ I,J\ge1,\,I+J=A \\ K+L=B}}
\frac{(w_{e_V}+w_e)A+w_{e_V}(B+2)}{\;\dis\Big\{(w_{e_V}+w_e)I+w_{e_V}(K+1)\Big\}\Big\{(w_{e_V}+w_e)J+w_{e_V}(L+1)\Big\}\;} \notag \\
&=
(w_{e_V}+w_e)\cdot
\frac{\vol(H)}{\;\dis\max_{\substack{I,J\in\N,\,K,L\in\N_0 \\ I+J=A \\ K+L=B}}\Big\{(w_{e_V}+w_e)I+w_{e_V}(K+1)\Big\}\Big\{(w_{e_V}+w_e)J+w_{e_V}(L+1)\Big\}\;}. \label{Fig3 2}
\end{align}

Therefore, by \autoref{flow of C}, \eqref{Fig3 1} and \eqref{Fig3 2}, we obtain \eqref{pp kxy weighted}.

%%%%%%%%%%%%%%%%%%%%%%%%%%%%%%%%%%%%%%%%%%%%%%%%%%%%%%%%%%%%%%%%%%%%%%%%%%
%%%%%%%%%%%%%%%%%%%%%%%%%%%%%%%%%%%%%%%%%%%%%%%%%%%%%%%%%%%%%%%%%%%%%%%%%%

\end{document}